\documentclass{article}
\usepackage{latexsym,amsfonts,amsthm,amsmath,amscd,amssymb}
\usepackage[dvips]{graphicx}

\newtheorem{lemma}{Lemma}[section]
\newtheorem{thm}[lemma]{Theorem}
\newtheorem{rem}[lemma]{Remark}
\newtheorem{prop}[lemma]{Proposition}

\newtheorem{example}[lemma]{Example}
\newtheorem{defn}[lemma]{Definition}

\newcommand\matR{{\mathbb{R}}}
\newcommand\matN{{\mathbb{N}}}

\renewcommand{\hbar}{{\overline{h}}}

\newfont{\Got}{eufm10 scaled 1200}

\newcommand\calD{{\mathcal D}}

\begin{document}

\title{Diffeological gluing of vector pseudo-bundles and pseudo-metrics on them}

\author{Ekaterina~{\textsc Pervova}}

\maketitle

\begin{abstract}
\noindent Although our main interest here is developing an appropriate analog, for diffeological vector pseudo-bundles, of a Riemannian metric, a significant portion is dedicated to continued study of the 
gluing operation for pseudo-bundles introduced in [9]. We give more details regarding the behavior of this operation with respect to gluing, also providing some details omitted from [9], and pay more attention 
to the relations with the spaces of smooth maps. We also show that a usual smooth vector bundle over a manifold that admits a finite atlas can be seen as a result of a diffeological gluing, and thus deduce 
that its usual dual bundle is the same as its diffeological dual. We then consider the notion of a pseudo-metric, the fact that it does not always exist (which seems to be related to non-local-triviality condition), 
construction of an induced pseudo-metric on a pseudo-bundle obtained by gluing, and finally, the relation between the spaces of all pseudo-metrics on the factors of a gluing, and on its result. We conclude by
commenting on the induced pseudo-metric on the pseudo-bundle dual to the given one.

\end{abstract}

\section*{Introduction}

For anyone who comes into contact with diffeology \cite{iglesiasBook} (meant initially as an extension of differential geometry, although as of now, other opinions exist) it does not take long to notice that 
a great number of typical objects do not admit obvious counterparts; there is obviously a diffeological space as the base object, and there are also various types of it endowed with an extra, algebraic, 
structure, starting with a natural concept of a diffeological vector space and proceeding towards the notion of a diffeological group, diffeological algebra, and so on. But when we try to endow these objects 
with further structures, even very simple ones, we discover that it is not possible to do so.

The specific instance that we have in mind at the moment is that of a scalar product on a finite-dimensional diffeological vector space. Surprising in its simplicity, it excellently illustrates what has been said 
in the previous paragraph. Specifically, it has been known for some time (see \cite{iglesiasBook}, Ex. 70 on p. 74) that a finite-dimensional diffeological vector space admits a smooth (with respect to its 
diffeology) scalar product if and only if it is a standard space, which means that, as a vector space, it is isomorphic to $\matR^n$ for some $n$, and its diffeology consists of all usual smooth maps. But for 
a single vector space this issue is easily resolved: one looks (as was done, for instance, in \cite{pseudometric}) for a kind of minimally degenerate smooth symmetric bilinear form on it; there is always one, 
of rank equal to the dimension of the so-called diffeological dual \cite{wu} of the space in question.

The next natural step, then, is to pose the same question for a possible diffeological counterpart of the notion of a Riemannian metric. The corresponding issues become two-fold then: for one thing, there is 
not yet a standard theory of tangent spaces in diffeology (but some constructions do exist, see, for instance, \cite{CWtangent}, as well as the previous \cite{HeTangent}, and other references therein), so we 
are left with looking for a diffeological version of a diffeological metric on an abstract (diffeological version of) vector bundle.

This version appeared originally in \cite{iglFibre} (by the author's own admission, the same material is better explained in Chapter 8 of \cite{iglesiasBook}; but the former is the original source). It was later 
utilized in \cite{vincent}, under the name of a \emph{regular vector bundle}, and then in \cite{CWtangent}, where it is called a \emph{diffeological vector space over $X$}, $X$ being the notation for the base 
space. We call it a \emph{diffeological vector pseudo-bundle}, as was already done in \cite{pseudobundles}. This concept is an obvious (it is asked that all the operations be smooth in the diffeological sense) 
extension to the diffeological context of the usual notion of a vector bundle, except that there is no requirement of local triviality. This allows to treat objects that not only carry an unusual smooth structure (that 
is, a diffeology) but also ones that, from a topological point of view, have a more complicated local structure than that of a Euclidean space.

For objects such as these, there are well-defined notions of the tensor product pseudo-bundle and of the dual pseudo-bundle, see \cite{vincent}. So a diffeological metric can of course be defined as a smooth 
section of the tensor product of two duals, symmetric at each point. However, due to what has been said about single vector spaces, the value at a given point cannot, in general, be a scalar product, unless 
the fibre at this specific point is standard (generally not the case). In this paper we discuss what happens if pointwise we try to take the minimally degenerate value of the prospective section (the approach 
initiated in \cite{pseudobundles}).

\paragraph{Acknowledgements} Various parts of this work emerged inside of several others --- and then moved in here, when it became more sensible for them to do so. I must admit that I maybe would not
have kept sensible enough in the meantime if it were not for my colleague and role model Prof. Riccardo Zucchi.

\section{Definitions needed}

In order to make the paper self-contained, we collect here all the definitions that are used in the rest of the paper.

\subsection{Diffeological spaces}

The basic notion for diffeology is that of a \emph{diffeological space}, that is, a set endowed with a \emph{diffeology}, a collection of maps called \emph{plots} that play the role of a (version of a) smooth 
structure.

\paragraph{Diffeological spaces and smooth maps between them} A diffeological space is just a set $X$ endowed with a \emph{diffeology}, which is a collection of maps from usual domains to $X$; three 
natural conditions must be satisfied.

\begin{defn} \emph{(\cite{So2})} A \textbf{diffeological space} is a pair $(X,\calD_X)$ where $X$ is a set and $\calD_X$ is a specified collection, also called the \textbf{diffeology} of $X$, of maps $U\to X$ 
(called \textbf{plots}) for each open set $U$ in $\matR^n$ and for each $n\in\matN$, such that for all open subsets $U\subseteq\matR^n$ and $V\subseteq\matR^m$ the following three conditions are 
satisfied:
\begin{enumerate}
\item (The covering condition) Every constant map $U\to X$ is a plot;
\item (The smooth compatibility condition) If $U\to X$ is a plot and $V\to U$ is a smooth map (in the usual sense) then the composition $V\to U\to X$ is also a plot;
\item (The sheaf condition) If $U=\cup_iU_i$ is an open cover and $U\to X$ is a set map such that each restriction $U_i\to X$ is a plot then the entire map $U\to X$ is a plot as well.
\end{enumerate}
\end{defn}

A standard example of a diffeological space is a standard manifold whose diffeology consists of all usual smooth maps into it. However, some standard constructions of diffeology (appearing shortly below)
allow us to construct plenty of unusual diffeological spaces.

Let now $X$ and $Y$ be two diffeological spaces (obviously, this includes $Y=X$), and let $f:X\to Y$ be a set map. It is a \textbf{smooth map} (as one between diffeological spaces) if for every plot $p$ of $X$ 
the composition $f\circ p$ is a plot of $Y$.

\paragraph{The underlying topology} Let $X$ be a diffeological space; it carries a natural topology underlying its diffeological structure and called \textbf{D-topology}, introduced in \cite{iglFibre}. This 
topology is defined by the following condition: a set $Y\subset X$ is open for D-topology (and is called \textbf{D-open}) if and only if for every plot $p:U\to X$ of $X$ its pre-image $p^{-1}(Y)$ is open in
$U\subseteq\matR^m$ in the usual sense. In fact, it was shown in \cite{CSW_Dtopology} (Theorem 3.7) that it is sufficient (and, of course, necessary) that this condition hold for smooth curves only, that is, 
for plots defined on $\matR$ (or its subdomain).

\paragraph{The lattice of diffeologies} As can be expected, the set of all diffeologies on a fixed set $X$ is partially ordered by inclusion; if $\calD$ and $\calD'$ are two diffeologies on $X$ then we say that 
$\calD$ is \textbf{finer} than $\calD'$ if $\calD\subset\calD'$, and it is said to be \textbf{coarser} if the inclusion is \emph{vice versa}. It is also a fact (see \cite{iglesiasBook}, Chapter 1.25) that with respect to 
this partial ordering the lattice is closed; and this fact is frequently used when defining various types of diffeologies (as, indeed, we do almost immediately below), because the lattice property frequently 
extends to sub-families of diffeologies possessing such-and-such specified property.

\paragraph{Discrete diffeology and coarse diffeology} The smallest/finest of all possible diffeologies on a given set $X$ is usually called the \textbf{discrete diffeology}; it consists of all locally constant maps 
with values in $X$ and the D-topology underlying it is the usual discrete topology. The largest of all possible diffeologies is called the \textbf{coarse diffeology}; it consists of all possible maps into $X$.

\paragraph{Generated diffeology} A frequent way to define a diffeology on a given set $X$ is to consider a(n arbitrary) collection $\mathcal{A}=\{p_i:U\to X\}$ of maps with values in $X$, each of which is 
defined on a domain of some $\matR^n$; to each such collection there corresponds a diffeology on $X$ called the diffeology \textbf{generated by $\mathcal{A}$} and defined as the smallest diffeology 
containing all maps $p_i\in\mathcal{A}$ as plots. The plots of the generated diffeology can locally be described as follows: they either are constant or filter through a map in $\mathcal{A}$. Some frequent 
instances of generated diffeologies include those generated by just one plot (this is the case of many of our examples) or those defined on a specific class of domains (such as the so-called wire diffeology, 
generated by maps defined on one-dimensional domains, that we will encounter below).

\paragraph{Pushforwards and pullbacks} Let $X$ be a diffeological space, and let $Y$ be any set (with or without any extra structure). Suppose first that we have a map $f:X\to Y$; then $Y$ can be endowed
with a specific diffeology, called the \textbf{pushforward diffeology} (from $X$ via $f$). It is defined as the finest diffeology such that $f$ is smooth; its plots are precisely those maps that locally are 
pre-compositions of $f$ with plots of $X$. Suppose now that, \emph{vice versa}, we have a map $f:Y\to X$; in this case, $Y$ can be endowed with the \textbf{pullback diffeology}, which is the coarsest diffeology 
such that $f$ is smooth. Its plots are precisely those maps whose compositions with $f$ are plots of $X$.

\paragraph{Functional diffeology} If now $X$ and $Y$ are both diffeological spaces, we can consider the set $C^{\infty}(X,Y)$ of all smooth maps $X\to Y$. It carries a canonical diffeology called \textbf{the 
functional diffeology}; this is the coarsest diffeology for which the \emph{evaluation map}, the map $\mbox{\textsc{ev}}:C^{\infty}(X,Y)\times X\to Y$ with $\mbox{\textsc{ev}}(f,x)=f(x)$, is smooth.

\paragraph{Subset diffeology and quotient diffeology} Every subset $Y\subset X$ of a diffeological space $X$ carries a natural diffeology called \textbf{subset diffeology}; it is the coarsest diffeology for which 
the inclusion map $Y\hookrightarrow X$ is smooth. In other words, it is the set of all plots of $X$ whose images are contained in $Y$. Likewise, for every equivalence relation $\sim$ the quotient space 
$X/\sim$ carries the natural diffeology called (unsurprisingly) the \textbf{quotient  diffeology}; this is the pushforward of the diffeology of $X$ by the natural projection. Locally, it consists of the compositions 
of plots of $X$ with the quotient projection.

\paragraph{The disjoint union diffeology and the direct product diffeology} Given a collection $X_1,\ldots,X_n$ of diffeological spaces (it does not have to be finite, but for us this is sufficient), there are standard 
diffeologies on the disjoint union $\coprod_{i=1}^nX_i$ and on the direct product $\prod_{i=1}^n$, called, respectively, \textbf{the sum diffeology} and \textbf{the product diffeology}. The sum diffeology is the 
finest diffeology such that all the natural inclusions $X_i\hookrightarrow\coprod_{i=1}^nX_i$ for $i=1,\ldots,n$ are smooth; locally any plot of it is a plot of one of the components of the disjoint union. The 
product diffeology is the coarsest diffeology such that for each $i=1,\ldots,n$ the natural projection $\prod_{i=1}^nX_i\to X$ is smooth; every plot of it locally is an $n$-tuple of plots of $X_i$, that is, has form 
$(p_1,\ldots,p_n)$ with $p_i$ a plot of $X_i$.

\subsection{Diffeological vector spaces}

Here we turn to the concept of a diffeological vector space (mostly, we speak about finite-dimensional ones, although various definitions apply elsewhere).

\subsubsection{Main definitions}

All the main concepts regarding vector spaces admit their diffeological counterparts, usually by a trivial extension, although sometimes with non-trivial implications.

\paragraph{A diffeological vector space} A \textbf{diffeological vector space} is a (real) vector space $V$ endowed with a \textbf{vector space diffeology}, that is, any diffeology for which the following two maps 
are smooth: the addition map $V\times V\to V$, where obviously $(v_1,v_2)\mapsto v_1+v_2$ and $V\times V$ carries the product diffeology, and the scalar multiplication map $\matR\times V\to V$, where 
$\matR$ carries the standard diffeology and, once again, $\matR\times V$ carries the product diffeology. All our spaces are finite-dimensional.

\paragraph{Fine diffeology} The set of all vector space diffeologies on a given vector space $V$ is quite large (although there are quite a few diffeologies that are not vector diffeologies\footnote{Two examples 
in case $V=\matR^n$ are the discrete diffeology, for which the scalar multiplication is not smooth, and the already-mentioned wire diffeology, where the addition is not smooth.}). The finest of them is called
the \textbf{fine diffeology}, and a vector space endowed with its fine diffeology is called a \textbf{fine vector space}.

\paragraph{Subspaces and quotient spaces} Any vector subspace $W\leqslant V$ of a diffeological vector space $V$ is itself a diffeological vector space for the subset diffeology; the latter is therefore the
canonical diffeology on $W$, and unless specified otherwise, any subspace is automatically considered with its subset diffeology. Likewise, the usual quotient space $V/W$ is a diffeological vector space for 
the quotient diffeology; it is canonically endowed with the latter.

\paragraph{Linear maps and smooth linear maps} Surprisingly at least for a finite-dimensional case (this is my opinion), a diffeological vector space in general admits linear maps that are not smooth. More
precisely, for every linear map to be smooth a finite-dimensional space must necessarily be standard, that is, diffeomorphic to the standard $\matR^n$ for the appropriate $n$. As an illustration, we add a 
simple example: if $V$ is identified with $\matR^n$ as a vector space and $V$ carries the vector space diffeology generated by the plot $\matR\ni x\to |x|e_n$ (in place of $|x|$ we could take
any not-everywhere-smooth function) then the corresponding dual (in the usual sense) function $e^n:V\to\matR$ is not smooth (for the chosen diffeology of $V$ and the standard diffeology of $\matR$).

\paragraph{Diffeological dual} The above example also shows that the so-called diffeological dual (\cite{vincent}, \cite{wu}) of a diffeological vector space (of finite dimension) has in general smaller 
dimension than the space itself. The \textbf{diffeological dual} $V^*$ of $V$ is the space of all smooth $\matR$-valued linear maps on $V$; it is not hard to find examples of spaces with trivial diffeological 
dual and not-too-large a diffeology. Indeed, it suffices to extend the example that appears in the preceding paragraph, endowing $\matR^n$ with the vector space diffeology generated by the $n$ plots 
$\matR\ni x\to |x|e_i$ for $i=1,\ldots,n$. No element $e^i$ of the canonical dual basis of $\matR^n$ is smooth for this diffeology, so the diffeological dual is trivial in this case, although the diffeology in question 
is a rather specific one.

\paragraph{Euclidean structure and pseudo-metrics} A \textbf{scalar product} on a diffeological vector space $V$ is a usual scalar product on the underlying vector space which is also smooth with respect 
to the diffeology of $V$. However (see \cite{iglesiasBook}), in the finite-dimensional case this concept does not make much sense, meaning that unless $V$ is a standard space, it does not carry any scalar 
product; in other words, if $V$ is finite-dimensional and is not diffeomorphic to some standard $\matR^n$ then every smooth symmetric bilinear form on $V$ is degenerate. The degree of degeneracy, 
unsurprisingly, is measured by the difference between the dimension of $V$ and that of its diffeological dual; that is to say, the maximal rank that a smooth symmetric bilinear form on $V$ might have is 
$\dim(V^*)$. It is easy to see that this rank is always achieved; any smooth symmetric bilinear form on $V$ of rank $\dim(V^*)$ is called a \textbf{pseudo-metric} on $V$.

\subsubsection{Operations with diffeological vector spaces and their properties}

We now recall what becomes of the usual operations on vector spaces in the context of diffeological vector spaces.

\paragraph{Direct sum} Let $V$ and $W$ be two diffeological vector spaces. Their usual direct sum $V\oplus W$ (which is in bijective correspondence with the direct product $V\times W$) is endowed with
the \emph{product} diffeology, which however takes the name of the \textbf{direct sum diffeology}. Note that interesting phenomena might happen when we consider a single (finite-dimensional)
diffeological vector space $V$ and its decomposition (in the sense of just vector spaces) into a direct sum of two of its subspaces, $V=V_1\oplus V_2$. Note that if $V_1$ and $V_2$ are considered with
the ir subset diffeologies then the direct sum diffeology on $V_1\oplus V_2$ might be strictly finer than the initial diffeology of $V$ (see \cite{pseudometric} for a specific example of this).

\paragraph{Tensor product} Let $V_1,\ldots,V_n$ be diffeological vector spaces. Their usual tensor product as vector spaces carries a natural \textbf{tensor product diffeology} (see \cite{vincent} and
\cite{wu}), which is defined as the pushforward, by the universal map $V_1\times\ldots\times V_n\to V_1\otimes\ldots\otimes V_n$, of the vector space diffeology on the free product of $V_1,\ldots,V_n$ 
that is generated by the product diffeology on $V_1\times\ldots\times V_n$.

\paragraph{Multilinear maps as elements of a tensor product} The diffeological tensor product possesses the appropriate counterpart of the usual universal property (see \cite{vincent}, Theorem 2.3.5),
which means the following: if $V_1,\ldots,V_n,W$ are diffeological vector spaces then the space of all smooth linear maps $V_1\otimes\ldots\otimes V_n\to W$, considered with the functional diffeology, 
is diffeomorphic to the space of all smooth multilinear maps $V_1\times\ldots\times V_n\to W$ (if it, too, is endowed with the functional diffeology). In particular, the space of all smooth bilinear $\matR$-valued 
maps on a given diffeological vector space $V$ is diffeomorphic to $(V\otimes V)^*$; we also note that the latter is diffeomorphic, as is usually the case, to $V^*\otimes V^*$.

\subsection{Diffeological vector pseudo-bundles}

We now turn to the object which is central to our discussion. Indeed, although a notion of metric (or its best-possible substitute, as in our case) would usually be defined on some kind of a tangent bundle, 
there is not yet a standard theory of such for diffeological spaces, although there do exist various \emph{ad hoc} treatments and proposals; in particular, the notion of the \emph{internal tangent bundle} (see 
\cite{CWtangent}) is now emerging prominently. Our solution to this matter is to consider a diffeological notion of a metric (that we call a \emph{pseudo-metric}), or whatever can reasonably be considered such, 
on abstract diffeological vector pseudo-bundles, on the assumption that whatever a proper notion of a tangent bundle for a diffeological space would reveal itself to be, it should be an instance of such a 
pseudo-bundle.

\subsubsection{Main definitions for pseudo-bundles}

We now recall, as briefly as possible, the main definitions and some facts regarding pseudo-bundles, concentrating on the aspects that we will need in what follows.

\paragraph{What is a diffeological vector pseudo-bundle} Let $\pi:V\to X$ be a smooth surjective map between two diffeological spaces $V$ and $X$. It is called a \textbf{diffeological vector pseudo-bundle}
if for every $x\in X$ the pre-image $\pi^{-1}(x)$ is endowed with a vector space structure such that the induced operations $V\times_X V\to V$ of fibrewise addition and $\matR\times V\to V$ of scalar
multiplication, as well as the zero section, are smooth, for the subset diffeology on $V\times_X V\subset V\times V$ and the product diffeologies on $V\times V$ and $\matR\times V$ (in the latter case
$\matR$ carries the standard diffeology). A usual vector bundle of course fits this definition, but there are many examples (some of which arise in independent contexts, see below) that are not vector
bundles in the usual sense. This has in part to do with the fact that $X$ and $V$, as diffeological spaces, do not have to be smooth, or even topological, manifolds, but more importantly, the map $\pi$
does not have to be locally trivial, as illustrated below.

\paragraph{The absence of local trivializations} As has just been mentioned, diffeological vector pseudo-bundles frequently lack local trivializations; there are two main ways for this to happen, as we
now explain.

The first way is that the pseudo-bundle in question could be locally trivial, or even trivial, from the point of view of the underlying topology, but not locally trivial from the diffeological point of view. An instance 
of this is easy to construct by taking the usual projection $\matR^2\to\matR$ of $\matR^2$ onto its first coordinate, and endowing the target $\matR$ with its standard diffeology and the source $\matR^2$ with 
the smallest diffeology generated by the map $\matR^2\ni(u,v)\mapsto(u,u|v|)$ with respect to which the projection becomes a pseudo-bundle. In this case, the subset diffeology on each fibre, except the one 
over the origin of the target $\matR$, is the vector space diffeology on $\matR$ generated by the map $v\mapsto |v|$ (in particular, it is non-standard). On the other hand, the fibre over the origin carries 
the standard diffeology, so in the neighbourhood of $0\in\matR$ the pseudo-bundle is non-trivial from the diffeological point of view, although it is trivial topologically.

The other thing that might happen is that a diffeological pseudo-bundle may have fibres of different dimension; in the neighbourhood of a point where the dimension changes the pseudo-bundle is of course 
non-trivial (this will be better explained via the concept of diffeological gluing, see below).

\paragraph{Generating a diffeology on the total space of a pseudo-bundle} Especially when dealing with examples, we will frequently encounter the following situation. Let $X$ be a diffeological space, 
and let $\pi:V\to X$ be a surjective map onto $X$ such that the pre-image $\pi^{-1}(x)$ of every point $x\in X$ carries a (finite-dimensional) vector space structure. We note first that there is obviously a 
diffeology on $V$ with respect to which $\pi$ is smooth: it suffices to take the pullback diffeology. This will be a pseudo-bundle diffeology (that is, it will make the operations smooth). Indeed, by Proposition 
4.16 of \cite{CWtangent}, any given diffeology on $V$ with respect to which $\pi$ is smooth, can be expanded to a pseudo-bundle diffeology, and since the pullback diffeology is the coarsest one for which 
$\pi$ is smooth, it will coincide with such expansion, and so be a pseudo-bundle diffeology itself.

Now, by a similar reasoning we can consider a \textbf{pseudo-bundle diffeology on $V$ generated by a given set $\mathcal{A}$ of plots}, provided that the composition of each plot in the set with $\pi$ is
a plot of $X$. If this requirement is satisfied then it suffices to take the diffeology of $V$ generated (in the sense of just generated diffeology) by $\mathcal{A}$. Since the map $\pi$ is smooth with respect to 
this diffeology, we can apply again Proposition 4.16 of \cite{CWtangent} to obtain a pseudo-bundle diffeology on $V$; this is going to be the smallest pseudo-bundle diffeology containing all maps in 
$\mathcal{A}$ as plots.

\paragraph{Notation for topologically trivial pseudo-bundles} Before proceeding, we give brief explanation of the \emph{ad hoc} notation that we will use to describe some examples to be seen below. The
notation has no pretence of generality; it applies to the pseudo-bundles whose underlying topological structure is that of the projection $\matR^n\to\matR^k$ to the first $k$ coordinates of some $\matR^n$ 
(with, obviously, $n>k$). However, already this very simple fibration yields interesting examples from the diffeological point of view, by choosing different pseudo-bundle diffeologies on the total space 
$\matR^n$ (whereas the base space, $\matR^k$, will usually be considered with its standard diffeology).

The vector space structure on the fibres of the just-mentioned projection comes from the second factor in the obvious decomposition as a direct product $\matR^n=\matR^k\times\matR^{n-k}$. Thus, in
particular, the elements of the dual space of the fibre at a given point $x\in\matR^k$ can be written in the form $(x,a_{k+1}e^{k+1}+\ldots+a_ne^n)$ for some real numbers $a_{k+1},\ldots,a_n$. As can 
be expected from what we have said already about diffeological duals, in general not all choices of these coefficients would give a smooth linear map on a given fibre; the point at this moment is however 
that all of them can be written in such a form.

Similarly, on appropriate occasions we can use this notation to describe the main object of our interest, a smooth bilinear form (hopefully one that comes as close as possible to a kind of a metric; we will see 
this later) on a pseudo-bundle as above. Such a form is, as usual, a smooth section of the tensor square of its dual, so we can write it as
$$x\mapsto(x,\sum_{i,j=k+1}^n a_{ij}e^i\otimes e^j).$$ As in the case of the dual, in general (that is, unless the diffeology is the standard one) there are restrictions on the coefficients $a_{ij}$ to ensure that 
the section is smooth (these are seen on the case-by-case basis). Besides, we will always deal with symmetric forms, so of course we'll have $a_{ij}=a_{ji}$.

\subsubsection{Vector bundle operations on pseudo-bundles}

In what follows we will need the following operations on pseudo-bundles: taking a sub-bundle, taking a quotient bundle, direct sum of pseudo-bundles, tensor product of pseudo-bundles, and taking the 
dual bundle. Since we cannot use local trivializations, as is standard, to define these operations, we briefly recall their definitions in the diffeological case.

\paragraph{Sub-bundles} A \textbf{sub-bundle} of diffeological vector pseudo-bundle $\pi:V\to X$ is any subset $W\subset V$ such that for every $x\in X$ the intersection $W\cap\pi^{-1}(x)$ is a vector 
subspace of $\pi^{-1}(x)$; the subset $W$ is equipped with the subset diffeology of $V$ and the restriction $\pi|_W$ of the pseudo-bundle map $\pi$. The map $\pi|_W:W\to X$ thus obtained is a
diffeological vector pseudo-bundle on its own; in particular, it is a surjective map, since the assumption that $W\cap\pi^{-1}(x)$ be always a vector subspace means in particular that it is always non-empty.

\paragraph{Quotient pseudo-bundles} Suppose now that we have a pseudo-bundle $\pi:V\to X$ and a sub-bundle $W\subset V$. This sub-bundle defines the obvious (it is defined separately on each
single fibre by the usual quotienting over the vector subspace $W\cap\pi^{-1}(x)$) equivalence relation on $V$; the quotient space $Z=V/W$, endowed with the quotient diffeology and the induced
projection onto $X$ is again a diffeological vector pseudo-bundle, which is called the \textbf{quotient pseudo-bundle}.

\paragraph{The direct sum} Let $\pi_1:V_1\to X$ and $\pi_2:V_2\to X$ be two diffeological vector pseudo-bundles with the same base space. Their \textbf{direct sum} is defined as 
$V_1\oplus V_2:=V_1\times_X V_2=\{(v_1,v_2)\,|\,\pi_1(v_1)=\pi_2(v_2)\}\subset V_1\times V_2$, endowed with the factorwise operations and the subset diffeology relative to the product diffeology on 
$V_1\times V_2$. The projection onto $X$ is obvious and is denoted by $\pi_1\oplus\pi_2$; and it is rather easy to see (the details are available in \cite{vincent}) that $\pi_1\oplus\pi_2:V_1\oplus V_2\to X$ 
is again a diffeological vector pseudo-bundle, with the fibre $(\pi_1\oplus\pi_2)^{-1}(x)=\pi_1^{-1}(x)\oplus\pi_2^{-1}(x)$ for any $x\in X$.

\paragraph{The tensor product} Once again, we choose a(n arbitrary) diffeological space $X$ and consider two (finite-dimensional) diffeological vector pseudo-bundles over it, $\pi_1:V_1\to X$ and
$\pi_2:V_2\to X$. The \textbf{tensor product} $\pi_1\otimes\pi_2:V_1\otimes V_2\to X$ of these two pseudo-bundles is by definition the quotient pseudo-bundle of the direct product pseudo-bundle 
$V_1\times_X V_2\to X$ over its sub-bundle $W$ given by the condition that its fibre $W_x$ at any $x\in X$ is the kernel of the usual universal map  
$\pi_1^{-1}(x)\times\pi_2^{-1}(x)\to\pi_1^{-1}(x)\otimes\pi_2^{-1}(x)$; as we have already mentioned, any collection of vector subspaces, taken one for each fibre, defines a sub-bundle, so $W$, and the
quotient pseudo-bundle over it, are well-defined. It is also clear from this definition that $(\pi_1\otimes\pi_2)^{-1}(x)=\pi_1^{-1}(x)\otimes\pi_2^{-1}(x)$ for every $x\in X$.

\paragraph{The dual bundle} Let $\pi:V\to X$ be a diffeological vector pseudo-bundle; to define its \textbf{dual pseudo-bundle} (for brevity, more often we will just say \emph{dual bundle}) we first set $V^*$ 
to be the formal union of the diffeological duals of all fibres of $V$, \emph{i.e.}, $V^*=\cup_{x\in X}(\pi^{-1}(x))^*$. This union is naturally equipped with the obvious projection onto $X$; we wish to endow it 
with a diffeology (which, recall, would also imply a topological structure), such that, first, the projection onto $X$ be smooth, and, second, the subset diffeology on any its component $(\pi^{-1}(x))^*$ coincide 
with its existing (functional) diffeology as the dual of $\pi^{-1}(x)$. The existence of this diffeology, and an explicit characterization of its plots were given in \cite{vincent}, Definition 5.3.1 and Proposition 5.3.2. 
This description of plots is as follows.

Let $\pi:V\to X$ be a diffeological vector pseudo-bundle, and let $\pi^*:V^*\to X$ be its dual pseudo-bundle. A map $p:\matR^k\supset U\to V^*$ is a plot of $V^*$ if and only if for every plot
$q:\matR^{k'}\supset U'\to V$ of $V$ the map $(u,u')\mapsto p(u)(q(u'))$ defined on the set $Y=\{(u,u')\,|\,\pi^*(p(u))=\pi(q(u))\}\subset U\times U'$ (and taking values in $\matR$) is smooth for the subset 
diffeology on $Y\subset\matR^{k+k'}$ and the standard diffeology of $\matR$. It is easy to notice that this is a fibrewise extension of the characterization of the functional diffeology on individual fibres.

\subsubsection{Assembling more complicated pseudo-bundles}

As we have already pointed out, diffeological pseudo-bundles lack local trivializations, and rightly so, since they may not be locally  trivial. A partial attempt to compensate for this absence could be
to consider a kind of assembling, which later on we will call \emph{gluing}, of more complicated pseudo-bundles from simpler ones, such as for instance those that from the topological point of view
are just trivial fiberings of some $\matR^n$ over some $\matR^k$. While there is no reason to expect that this would describe any finite-dimensional diffeological vector pseudo-bundle (rather, given
the breadth of diffeology, we might expect the obvious), it could still be a way to single out, and render more manageable, some reasonable, and reasonably wide, classes of them.

\paragraph{Gluing of two diffeological spaces} The basic operation that defines this assembling is the \textbf{gluing of two diffeological spaces}. Let $X_1$ and $X_2$ be diffeological spaces, and let 
$f:X_1\supset Y\to X_2$ be a smooth (for the subset diffeology on $Y$) map; more often than not we will assume that it is a diffeomorphism with its image, but this is not necessary at the moment. We use 
the standard (topological) notion defining the result of \emph{gluing $X_1$ to $X_2$ along $f$}, as a set, by setting 
$$X_1\cup_f X_2=\left(X_1\sqcup X_2\right)/\sim,$$ where $\sim$ is the equivalence relation on the disjoint union given by $Y\ni y\sim f(y)$ (with all other points being equivalent to themselves only). 
The set $X_1\cup_f X_2$ is then endowed with the \textbf{gluing diffeology}, which is the obvious one: the quotient diffeology of the disjoint union diffeology on $X_1\sqcup X_2$.

\paragraph{Gluing of two pseudo-bundles} Let $\pi_1:V_1\to X_1$ and $\pi_2:V_2\to X_2$ be two diffeological vector pseudo-bundles. To perform a gluing on them, we just need two maps, one for the bases
and the other for the total spaces, with a number of conditions that ensure that the result of the gluing is again a diffeological vector pseudo-bundle (with the pseudo-bundle map arising naturally from the
two given ones, $\pi_1$ and $\pi_2$). These conditions are as follows. The gluing of the base spaces is carried out along a map $f:X_1\supset Y\to f(Y)\subset X_2$, which is required to be smooth
(for the subset diffeologies of its domain and co-domain); in most cases we will actually assume that it is a diffeomorphism with its image. The gluing of the total spaces is done along a smooth and
fibrewise linear lift $\tilde{f}:\pi_1^{-1}(Y)\to\pi_2^{-1}(f(Y))$.

Given the data listed in the above paragraph, we obtain two diffeological spaces, $V_1\cup_{\tilde{f}}V_2$ and $X_1\cup_f X_2$; the pseudo-bundle maps $\pi_1$ and $\pi_2$ induce in an obvious
manner a smooth surjection, denoted by $\pi_1\cup_{(\tilde{f},f)}\pi_2$, of the former onto the latter. Furthermore, each fibre of this surjection inherits a vector space structure from either $V_1$ or $V_2$, as 
relevant. We finally observe that the resulting map
$$\pi_1\cup_{(\tilde{f},f)}\pi_2:V_1\cup_{\tilde{f}}V_2\to X_1\cup_f X_2$$ is again a diffeological vector pseudo-bundle (see \cite{pseudobundles} for the formal proof). This pseudo-bundle is precisely what 
we mean by the result of \textbf{gluing the pseudo-bundle $\pi_1:V_1\to X_1$ to $\pi_2:V_2\to X_2$ along $(\tilde{f},f)$}.

\section{Operations and gluing}

The first results regarding the behavior of the vector bundle operations with respect to the gluing were given in \cite{pseudobundles}. Here we give details that were omitted from there, along with some new 
examples and observations.

\subsection{Sub-bundles and quotient pseudo-bundles}

Let $\pi_1:V_1\to X_1$ and $\pi_2:V_2\to X_2$ be two diffeological vector pseudo-bundles, and let $W_1\subset V_1$ and $W_2\subset V_2$ be their sub-bundles. Let $f:X_1\supset Y\to f(Y)\subset X_2$ be 
a smooth map, and let $\tilde{f}:\pi_1^{-1}(Y)\to\pi_2^{-1}(f(Y))$ be its smooth fibrewise linear lift.

\paragraph{Sub-bundles} What follows is an extended version of a lemma that appears in \cite{pseudobundles}.

\begin{lemma}
The image $\tilde{W}$ of $W_1\sqcup W_2$ in $V_1\cup_{\tilde{f}}V_2$ is a sub-bundle of the latter if and only if for every $y\in Y$ one of the following two cases occurs: either
$\tilde{f}(\pi_1^{-1}(y)\cap W_1)\leqslant\pi_2^{-1}(y)\cap W_2$ or $\tilde{f}(\pi_1^{-1}(y)\cap W_1)\geqslant\pi_2^{-1}(y)\cap W_2$ (the reverse inclusion).
\end{lemma}

\begin{proof}
The only condition that $W_1\cup_{\tilde{f}}W_2$ must satisfy in order to be a sub-bundle of $V_1\cup_{\tilde{f}}V_2$ is that its intersection with every fibre of $\pi_1\cup_{\tilde{f}}\pi_2$ be a vector subspace. 
It is also clear that outside of the domain of gluing this condition is automatic. It remains to notice that for any $y$ the domain of gluing, $(W_1\cup_{\tilde{f}}W_2)\cap(\pi_1\cup_{\tilde{f}}\pi_2)^{-1}(y)$ is a 
subspace in the fibre if and only if one of the two inclusions of the statement of the lemma holds.
\end{proof}

Let us now consider the question of the induced gluing of the sub-bundles. The main observation (a kind of non-commutativity for sub-bundles) is that for there to exist a well-defined induced gluing of 
$W_1$ to $W_2$ it is not sufficient that $\tilde{W}$ be a sub-bundle. Indeed, if at least at one point we have the strict inclusion $\tilde{f}(\pi_1^{-1}(y)\cap W_1)>\pi_2^{-1}(y)\cap W_2$ then the restriction 
of $\tilde{f}$ to $W_1$ does not generally take values in $W_2$. Therefore, $W_1\cup_{\tilde{f}|_{W_1}}W_2$ is well-defined (and is a sub-bundle of $V_1\cup_{\tilde{f}}V_2$) if and only if for every 
$y\in Y$ we have $\tilde{f}(\pi_1^{-1}(y)\cap W_1)\leqslant\pi_2^{-1}(y)\cap W_2$ (so the sub-bundle that we previously called $\tilde{W}$ coincides with $W_1\cup_{\tilde{f}|_{W_1}}W_2$, and the latter can 
be described also as the image of $W_1\sqcup W_2$ in $V_1\cup_{\tilde{f}}W_2$).

\paragraph{Quotient pseudo-bundles} Let us now consider $Z_1=V_1/W_1$ and $Z_2=V_2/W_2$; we denote the corresponding pseudo-bundle maps by $\pi_1^Z:Z_1\to X_1$ and $\pi_2^Z:Z_2\to X_2$
respectively. Notice that already at the level of individual fibres we must have
$$\tilde{f}(W_1\cap\pi_1^{-1}(y))\subset W_2\cap\pi_2^{-1}(y)\mbox{ for all }y\in Y$$ for there to be a well-defined induced gluing of $Z_1$ to $Z_2$. (Thus, according to what has been said in the previous 
paragraph we must have $W_1\cup_{\tilde{f}|_{W_1}}W_2=\tilde{W}$). If this condition holds, there is a well-defined induced map 
$\tilde{f}^Z:Z_1\supset\pi_1^{-1}(Y)/\left(\pi_1^{-1}(Y)\cap W_1\right)\to\pi_2^{-1}(f(Y))/\left(\pi_2^{-1}(f(Y))\cap W_2\right)\subset Z_2$, which is a lift of $f$ to the pseudo-bundles $Z_1$ and $Z_2$. Thus, we 
have a gluing of the quotient pseudo-bundles, with the result the pseudo-bundle
$$\pi_1^Z\cup_{(\tilde{f}^Z,f)}\pi_2^Z:Z_1\cup_{\tilde{f}^Z}Z_2\to X_1\cup_f X_2.$$

\begin{lemma}
The pseudo-bundle $Z_1\cup_{\tilde{f}^Z}Z_2$ is diffeomorphic to the quotient pseudo-bundle $(V_1\cup_{\tilde{f}}V_2)/(W_1\cup_{\tilde{f}|_{W_1}}W_2)$ via a diffeomorphism which is a lift of the 
identity morphism on the bases.
\end{lemma}

\begin{proof}
By construction, there is an identity between each fibre of the pseudo-bundle $Z_1\cup_{\tilde{f}^Z}Z_2$ and the corresponding fibre of $(V_1\cup_{\tilde{f}}V_2)/(W_1\cup_{\tilde{f}|_{W_1}}W_2)$. This
is trivial outside of the domain of gluing. If now $y\in Y\subset X_1\cup_f X_2$, we observe that the fibre of the former pseudo-bundle is $\pi_2^{-1}(f(y))/\left(\pi_2^{-1}(f(y))\cap W_2\right)$, and that of the 
latter pseudo-bundle is the same, because the fibre of the sub-bundle $W_1\cup_{\tilde{f}|_{W_1}}W_2$ at $y=f(y)$ is $\pi_2^{-1}(f(y))\cap W_2$.
\end{proof}

\subsection{Direct sum and gluing}

We now turn to the binary operation of taking the direct sum. It is natural to ask whether it commutes with gluing (in the sense that should be rather obvious, but we do state the question precisely
immediately below). 

\paragraph{The commutativity problem} Let $X_1$ and $X_2$ be two diffeological spaces. Over each of them, we consider two diffeological vector pseudo-bundles, $\pi_i:V_i\to X_i$ and $\pi_i':V_i'\to X_i$ 
for $i=1,2$. This allows to take the direct sum of the two pseudo-bundles over each $X_i$, obtaining the pseudo-bundles
$$\pi_1\oplus\pi_1':V_1\oplus V_1'\to X_1,\mbox{ and }\pi_2\oplus\pi_2':V_2\oplus V_2'\to X_2.$$

On the other hand, we consider also a fixed gluing of $X_1$ to $X_2$ and two pairs of gluings between the pseudo-bundles over them. The former is performed, as usual, along the map 
$f:X_1\supset Y\to f(Y)\subset X_2$; we glue the pseudo-bundle $V_1$ to $V_2$ along a smooth fibrewise linear lift $\tilde{f}:\pi_1^{-1}(Y)\to\pi_2^{-1}(f(Y))$ of $f$, and similarly, we glue $V_1'$ to $V_2'$ 
along another (smooth, fibrewise linear) lift $\tilde{f}':(\pi_1')^{-1}(Y)\to(\pi_2')^{-1}(f(Y))$ of the same $f$. This produces the following two pseudo-bundles:
$$\pi_1\cup_{(\tilde{f},f)}\pi_2:V_1\cup_{\tilde{f}}V_2\to X_1\cup_f X_2,\mbox{ and }\pi_1'\cup_{(\tilde{f}',f)}\pi_2':V_1'\cup_{\tilde{f}'}V_2'\to X_1\cup_f X_2.$$ Now, the latter two bundles are over the same base
space, so we can take their direct sum:
$$(\pi_1\cup_{(\tilde{f},f)}\pi_2)\oplus(\pi_1'\cup_{(\tilde{f}',f)}\pi_2'):(V_1\cup_{\tilde{f}}V_2)\oplus(V_1'\cup_{\tilde{f}}V_2')\to X_1\cup_f X_2.$$ We can also observe that also the two direct sum
pseudo-bundles can be glued along the map $f$ and its lift $\tilde{f}\oplus\tilde{f}':(\pi_1\oplus\pi_1')^{-1}(Y)\to(\pi_2\oplus\pi_2')^{-1}(f(Y))$. This gluing produces the pseudo-bundle
$$(\pi_1\oplus\pi_2)\cup_{\tilde{f}\oplus\tilde{f}'}(\pi_1'\oplus\pi_2'):(V_1\oplus V_2)\cup_{\tilde{f}\oplus\tilde{f}'}(V_1'\oplus V_2')\to X_1\cup_f X_2.$$ The precise form of the question whether the gluing
commutes with the direct sum is the following one: does there exist a diffeomorphism
$$\Phi:(V_1\cup_{\tilde{f}}V_2)\oplus(V_1'\cup_{\tilde{f}}V_2')\to(V_1\oplus V_2)\cup_{\tilde{f}\oplus\tilde{f}'}(V_1'\oplus V_2')$$ such that
$$(\pi_1\cup_{(\tilde{f},f)}\pi_2)\oplus(\pi_1'\cup_{(\tilde{f}',f)}\pi_2')=\left((\pi_1\oplus\pi_2)\cup_{\tilde{f}\oplus\tilde{f}'}(\pi_1'\oplus\pi_2')\right)\circ\Phi?$$

\paragraph{The answer} The answer is in the affirmative (it was already stated in \cite{pseudobundles}) by the properties of the direct sum; it suffices to examine the fibres at each point. Note that there is
an obvious identity for the fibres that are \emph{not} over points in the domain of gluing. There is actually the same kind of identity for fibres over the points in the domain of gluing (it is just a bit less obvious). 
Nevertheless, we omit a detailed consideration here (but we provide an extensive explanation in the case what follows, that of tensor product; the two explanations would be quite similar, so we omit the more 
evident one).

\subsection{Tensor product and gluing}

One (out of two) ingredients crucial for considering the existence questions related to our main subject (pseudo-metrics on pseudo-bundles) is that of the behavior of the tensor product with respect to 
gluing. The issue is obvious (we state it below in a precise form): if we have two pseudo-bundles, over the same base, each of which is the result of a gluing (with the same gluing on the base), we can take 
their tensor product; but we can also first take the tensor products of the factors and then perform an induced gluing of the results. It turns out that the end result is the same, but this is not quite obvious.

\paragraph{The statement of the problem} Let $\pi_1:V_1\to X_1$ and $\pi_1':V_1'\to X_1$, and $\pi_2:V_2\to X_2$ and $\pi_2':V_2'\to X_2$ be two pairs of diffeological vector pseudo-bundles, let
$f:X_1\supset Y\to f(Y)\subset X_2$ be a smooth invertible map with smooth inverse, and let $\tilde{f}:V_1\supset\pi_1^{-1}(Y)\to\pi_2^{-1}(f(Y))\subset V_2$ and
$\tilde{f}':V_1'\supset(\pi_1')^{-1}(Y)\to(\pi_2')^{-1}(f(Y))\subset V_2'$ be two smooth fibrewise linear lifts of it. These data allow to define the following two pairs of pseudo-bundles, the two possible tensor 
products in the first pair:
$$\pi_1\otimes\pi_1':V_1\otimes V_1'\to X_1\mbox{ and }\pi_2\otimes\pi_2':V_2\otimes V_2'\to X_2,$$ and then two gluings:
$$\pi_1\cup_{(\tilde{f},f)}\pi_2:V_1\cup_{\tilde{f}}V_2\to X_1\cup_f X_2\mbox{ and }\pi_1'\cup_{(\tilde{f}',f)}\pi_2':V_1'\cup_{\tilde{f}'}V_2'\to X_1\cup_f X_2.$$ We now consider the tensor product of the
pseudo-bundles in the second pair:
$$(\pi_1\cup_{(\tilde{f},f)}\pi_2)\otimes(\pi_1'\cup_{(\tilde{f}',f)}\pi_2'):(V_1\cup_{\tilde{f}}V_2)\otimes(V_1'\cup_{\tilde{f}'}V_2')\to X_1\cup_f X_2,$$ and then the induced gluing of the pseudo-bundles
of the first pair (this induced gluing is along the maps $f$ and $\tilde{f}\otimes\tilde{f}'$):
$$(\pi_1\otimes\pi_1')\cup_{\tilde{f}\otimes\tilde{f}'}(\pi_2\otimes\pi_2'):(V_1\otimes V_1')\cup_{\tilde{f}\otimes\tilde{f}'}(V_2\otimes V_2')\to X_1\cup_f X_2.$$ The question is then, whether the two
pseudo-bundles thus obtained, $(\pi_1\cup_{(\tilde{f},f)}\pi_2)\otimes(\pi_1'\cup_{(\tilde{f}',f)}\pi_2')$ and $(\pi_1\otimes\pi_1')\cup_{\tilde{f}\otimes\tilde{f}'}(\pi_2\otimes\pi_2')$, are the same (up to some 
natural pseudo-bundle diffeomorphism).

\paragraph{The answer} As in the case of the direct sum, the answer is in the affirmative. Before stating and proving the corresponding result, we introduce some technical notation (that will also be used 
elsewhere). Specifically, let
$$i_1:X_1\setminus Y\hookrightarrow X_1\sqcup X_2\to X_1\cup_f X_2,\,\,\,i_2:X_2\hookrightarrow X_1\sqcup X_2\to X_1\cup_f X_2,$$
$$j_1:V_1\setminus\pi_1^{-1}(Y)\hookrightarrow V_1\sqcup V_2\to V_1\cup_{\tilde{f}}V_2,\,\,\,j_2:V_2\hookrightarrow V_1\sqcup V_2\to V_1\cup_{\tilde{f}}V_2$$
be the obvious maps induced by the natural inclusions and the quotient projections $X_1\sqcup X_2\to X_1\cup_f X_2$ and $V_1\sqcup V_2\to V_1\cup_{\tilde{f}}V_2$. For the moment we only need 
to notice that the images $i_1(X_1\setminus Y)$ and $i_2(X_2)$ disjointly cover $X_1\cup_f X_2$, and the images $j_1(V_1\setminus\pi_1^{-1}(Y))$ and $j_2(V_2)$ disjointly cover $V_1\cup_{\tilde{f}}V_2$; 
their further properties are considered in Section \ref{natural:inductions:sect}.

\begin{thm}
Let $\pi_1:V_1\to X_1$ and $\pi_1':V_1'\to X_1$, and $\pi_2:V_2\to X_2$ and $\pi_2':V_2'\to X_2$ be two pairs of finite-dimensional diffeological vector pseudo-bundles. Let $f:X_1\supset Y\to f(Y)\subset X_2$ 
be a smooth map that is a diffeomorphism with its image, and let $\tilde{f}:V_1\supset\pi_1^{-1}(Y)\to\pi_2^{-1}(f(Y))\subset V_2$ and $\tilde{f}':V_1'\supset(\pi_1')^{-1}(Y)\to(\pi_2')^{-1}(f(Y))\subset V_2'$ be 
two smooth lifts of it that are linear on all fibres in their respective domains of definition. Then there exists a diffeomorphism
$$\Psi:(V_1\cup_{\tilde{f}}V_2)\otimes(V_1'\cup_{\tilde{f}'}V_2')\to(V_1\otimes V_1')\cup_{\tilde{f}\otimes\tilde{f}'}(V_2\otimes V_2')$$ such that
$$(\pi_1\cup_{(\tilde{f},f)}\pi_2)\otimes(\pi_1'\cup_{(\tilde{f}',f)}\pi_2')=\left((\pi_1\otimes\pi_1')\cup_{\tilde{f}\otimes\tilde{f}'}(\pi_2\otimes\pi_2')\right)\circ\Psi.$$
\end{thm}

\begin{proof}
Let us write for brevity
$$V':=(V_1\cup_{\tilde{f}}V_2)\otimes(V_1'\cup_{\tilde{f}'}V_2'),\,\,\,V'':=(V_1\otimes V_1')\cup_{\tilde{f}\otimes\tilde{f}'}(V_2\otimes V_2'),\mbox{ and }$$
$$\pi'=(\pi_1\cup_{(\tilde{f},f)}\pi_2)\otimes(\pi_1'\cup_{(\tilde{f}',f)}\pi_2'),\,\,\,\pi''=(\pi_1\otimes\pi_1')\cup_{\tilde{f}\otimes\tilde{f}'}(\pi_2\otimes\pi_2');$$ then both $V'$ and $V''$ fiber over $X_1\cup_f X_2$. Let 
us first describe a fibrewise bijection $\Psi:V'\to V''$ and then show that it is both ways smooth. The description obviously depends on the following three cases:\vspace{2mm}

\noindent\emph{Case 1:} $x\in X_1\setminus Y\subset X_1\cup_f X_2$. The fibre $(\pi')^{-1}(x)$ is the tensor product of the following two fibres: $(\pi_1\cup_{(\tilde{f},f)}\pi_2)^{-1}(x)$ and
$(\pi_1'\cup_{(\tilde{f}',f)}\pi_2')^{-1}(x)$. By the choice of $x$ the former coincides with $\pi_1^{-1}(x)$ and the latter, with $(\pi_1')^{-1}(x)$. Thus, for $x\in X_1\setminus Y$ we essentially have 
$(\pi')^{-1}(x)=\pi_1^{-1}(x)\otimes(\pi_1')^{-1}(x)$. On the other hand, the fibre $(\pi'')^{-1}(x)$, again by choice of $x$, is actually contained in $(V_1\otimes V_1')\subset V''$, thus it is also the tensor product 
of fibres $\pi_1^{-1}(x)$ and $(\pi_1')^{-1}(x)$. So we have $(\pi'')^{-1}(x)=\pi_1^{-1}(x)\otimes(\pi_1')^{-1}(x)$ as well. In particular, there is an obvious (identity) bijection $(\pi')^{-1}(x)\to(\pi'')^{-1}(x)$.
\vspace{2mm}

\noindent\emph{Case 2:}  $x\in X_2\setminus f(Y)\subset X_1\cup_f X_2$. This case is completely analogous to the previous one. The fibre $(\pi')^{-1}(x)$ is again the tensor product of
$(\pi_1\cup_{(\tilde{f},f)}\pi_2)^{-1}(x)$ and $(\pi_1'\cup_{(\tilde{f}',f)}\pi_2')^{-1}(x)$. Since $x\in X_2\setminus f(Y)$, we have $(\pi_1\cup_{(\tilde{f},f)}\pi_2)^{-1}(x)=\pi_2^{-1}(x)$ and
$(\pi_1'\cup_{(\tilde{f}',f)}\pi_2')^{-1}(x)=(\pi_2')^{-1}(x)$, so we have $(\pi')^{-1}(x)=\pi_2^{-1}(x)\otimes(\pi_2')^{-1}(x)$.

As for the fibre $(\pi'')^{-1}(x)$, it is contained in $(V_2\otimes V_2')\subset V_2''$, so we also have $(\pi'')^{-1}(x)=\pi_2^{-1}(x)\otimes(\pi_2')^{-1}(x)$. Thus, there is again a what is essentially the identity map
$(\pi')^{-1}(x)\to(\pi'')^{-1}(x)$. We should also note that in the two cases just considered the identity established is smooth for the relevant subset diffeologies.\vspace{2mm}

\noindent\emph{Case 3:} $x\in X_1\cap X_2\subset X_1\cup_f X_2$, meaning that $x$ writes both as $x=y\in Y$ and $x=f(y)\in f(Y)$. This is the only case that might raise some doubt. The fibre $(\pi')^{-1}(x)$ 
is the tensor product of $(\pi_1\cup_{(\tilde{f},f)}\pi_2)^{-1}(x)$ and $(\pi_1'\cup_{(\tilde{f}',f)}\pi_2')^{-1}(x)$; the former of these two fibres is the quotient $\pi_1^{-1}(x)\cup_{\tilde{f}}\pi_2^{-1}(x)$ and as a single 
fibre coincides with $\pi_2^{-1}(x)$ (this being the target space of $\tilde{f}$). Similarly, the fibre $(\pi_1'\cup_{(\tilde{f}',f)}\pi_2')^{-1}(x)$ is the quotient $(\pi_1')^{-1}(x)\cup_{\tilde{f}}(\pi_2')^{-1}(x)$ and coincides
with $(\pi_2')^{-1}(x)$. Thus, we have a natural identification $(\pi')^{-1}(x)\to\pi_2^{-1}(x)\otimes(\pi_2')^{-1}(x)$.

We now consider the fibre $(\pi'')^{-1}(x)$. This fibre is obtained by identification, via map $\tilde{f}\otimes\tilde{f}'$ of the following two tensor products: $\pi_1^{-1}(x)\otimes(\pi_1')^{-1}(x)$ and 
$\pi_2^{-1}(x)\otimes(\pi_2')^{-1}(x)$. The target space of the map $\tilde{f}\otimes\tilde{f}'$ is the second of these two products, which means that $(\pi'')^{-1}(x)$ coincides with it,
$(\pi'')^{-1}(x)=\pi_2^{-1}(x)\otimes(\pi_2')^{-1}(x)$. This gives us the third, and final, necessary identification $(\pi')^{-1}(x)\to(\pi'')^{-1}(x)$.

The above three are all the possibilities, so we have an obvious bijection $\Psi$, which is also a fibrewise isomorphism of vector spaces. It remains to consider its smoothness.

We summarize now the defining identities of $\Psi$:
$$\Psi:\left\{\begin{array}{ll}
(\pi')^{-1}(x)=\pi_1^{-1}(x)\otimes(\pi_1')^{-1}(x)=(\pi'')^{-1}(x)& \mbox{for }x\in X_1\setminus Y\subset X_1\cup_f X_2 \\
(\pi')^{-1}(x)=\pi_2^{-1}(x)\otimes(\pi_2')^{-1}(x)=(\pi'')^{-1}(x) & \mbox{for }x\in X_2\setminus f(Y)\subset X_1\cup_f X_2 \\
(\pi')^{-1}(x)=\pi_2^{-1}(x)\otimes(\pi_2')^{-1}(x)=(\pi'')^{-1}(x) & \mbox{for }x\in X_1\cap X_2\subset X_1\cup_f X_2.
\end{array}\right.$$
By examining these identities, the smoothness of $\Psi$ follows from the definition of the gluing diffeology. As we have already mentioned, outside of the domain of gluing the smoothness is automatic, since 
$\Psi$ is essentially the identity map.

Let us consider it on the domain of gluing. Let $p':U\to V'$ be a plot of $V'$ whose range intersects $(\pi_1\cup_{(\tilde{f},f)}\pi_2)^{-1}(Y)\otimes(\pi_1'\cup_{(\tilde{f}',f)}\pi_2')^{-1}(Y)$. We can always assume 
that $U$ is connected and furthermore (by definition of the tensor product diffeology) is small enough so that $p'$ is represented by a formal sum of pairs of plots $\sum_i(q_i,q_i')$, where each $q_i$ is a plot 
of $V_1\cup_{\tilde{f}}V_2$, and each $q_i'$ is a plot of $V_1'\cup_{\tilde{f}'}V_2'$. It is the sufficient to consider the case of a single pair $(q,q')$. Now, by the assumption that $U$ is connected, and by the 
definition of the gluing diffeology, there are two cases:
\begin{enumerate}
\item $q$ and $q'$ lift respectively to plots $q_1$ of $V_1$ and $q_1'$ of $V_1'$, such that 
$$q(u)=\left\{\begin{array}{ll} j_1(q_1(u)) & \mbox{if }q_1(u)\in V_1\setminus\pi_1^{-1}(Y), \\ j_2(\tilde{f}(q_1(u))) & \mbox{if } q_1(u)\in\pi_1^{-1}(Y), \end{array}\right.\,\,\mbox{ and }\,\, 
q'(u)=\left\{\begin{array}{ll} j_1'(q_1'(u)) & \mbox{if }q_1'(u)\in V_1'\setminus(\pi_1')^{-1}(Y), \\ j_2'(\tilde{f}'(q_1'(u))) & \mbox{if } q_1'(u)\in(\pi_1')^{-1}(Y), \end{array}\right.;$$
\item $q$ and $q'$ lift respectively to plots $q_2$ of $V_2$ and $q_2'$ of $V_2'$ such that $\mbox{Range}(q_2)\subseteq\tilde{f}(\pi_1^{-1}(Y))$, $\mbox{Range}(q_2')\subseteq\tilde{f}'((\pi_1')^{-1}(Y))$, 
and moreover, $q=j_2\circ q_2$ and $q'=j_2'\circ q_2'$.
\end{enumerate}
In the first of these cases, the pair $(q,q')$ writes as $(j_1,j_1')\circ(q_1,q_1')$ over $i_1(X_1\setminus Y)$, and as $(j_2\circ\tilde{f},j_2'\circ\tilde{f}')\circ(q_1,q_1')$ over the domain of gluing $i_2(f(Y))$. This 
means that 
$$\Psi\circ p'=
\left\{\begin{array}{ll} (j_1\otimes j_1')\circ(q_1\otimes q_1') & \mbox{over }i_1(X_1\setminus Y) \\ ((j_2\circ\tilde{f})\otimes(j_2'\circ\tilde{f}'))\circ(q_1\otimes q_1') & \mbox{over }i_2(f(Y)), \end{array}\right.$$ which 
precisely one of the two possible (local) forms of a plot of $V''$. Likewise, in the second case, the composition $\Psi\circ p'$ has form $(j_2\otimes j_2')\circ(q_2\otimes q_2')$, the other local form of a plot of 
$V''$. 

All cases having thus been considered, we conclude that $\Psi$ is indeed a smooth map. The smoothness of its inverse is checked in much the same way, and this completes the proof.
\end{proof}

\subsection{The dual pseudo-bundle of $\pi_1\cup_{(\tilde{f}^*,f)}\pi_2:V_1\cup_{\tilde{f}}V_2\to X_1\cup_f X_2$}

This case presents the most substantial differences, as generally, taking the dual pseudo-bundles does not commute with gluing; in fact, it appears that this is a condition that should be explicitly imposed.

\subsubsection{The statement of the commutativity problem}

Let $\pi_1:V_1\to X_1$ and $\pi_2:V_2\to X_2$ be two diffeological vector pseudo-bundles, which are glued along the maps $f:X_1\supset Y\to f(Y)\subset X_2$ and $\tilde{f}:\pi_1^{-1}(Y)\to\pi_2^{-1}(f(Y))$. 
We assume that $f$ is a diffeomorphism with its image (this assumption is necessary in this case, since we are also going to consider the gluing along the inverse of $f$) and that $\tilde{f}$ is a smooth lift of 
$f$ linear on each fibre. The result of gluing of $\pi_1:V_1\to X_1$ to $\pi_2:V_2\to X_2$ along the pair $(\tilde{f},f)$ is the pseudo-bundle
$$\pi_1\cup_{(\tilde{f},f)}\pi_2:V_1\cup_{\tilde{f}}V_2\to X_1\cup_f X_2,$$ with the dual bundle
$$(\pi_1\cup_{(\tilde{f},f)}\pi_2)^*:(V_1\cup_{\tilde{f}}V_2)^*\to X_1\cup_f X_2.$$
On the other hand, the pair $(\tilde{f},f)$ provides for a natural gluing between the dual bundles $\pi_1^*:V_1^*\to X_1$ and $\pi_2^*:V_2^*\to X_2$. This induced gluing is performed along the maps 
$\tilde{f}^*$ (the fibrewise dual map) and $f^{-1}:X_2\supset f(Y)\to Y\subset X_1$. Its result is the pseudo-bundle
$$\pi_2^*\cup_{(\tilde{f}^*,f)}\pi_1^*:V_2^*\cup_{\tilde{f}^*}V_1^*\to X_2\cup_{f^{-1}}X_1.$$

The commutativity problem is the natural question of whether the two pseudo-bundles thus obtained are diffeomorphic (as pseudo-bundles). Specifically, does there exist a diffeomorphism
$\Phi:(V_1\cup_{\tilde{f}}V_2)^*\to V_2^*\cup_{\tilde{f}^*}V_1^*$ such that
$$(\pi_1\cup_{(\tilde{f},f)}\pi_2)^*=\left(\pi_2^*\cup_{(\tilde{f}^*,f)}\pi_1^*\right)\circ\Phi?$$ Below we illustrate via an example that the answer is negative in general.

\subsubsection{The necessary condition for gluing-dual commutativity: $(\pi_2^{-1}(f(y)))^*=(\pi_1^{-1}(y))^*$ for all $y\in Y$}

The condition mentioned is based on the fact that the fibres in the result of a gluing of pseudo-bundle are the target spaces of the gluing map between the total spaces (else they are inherited from one of 
the factors in the gluing).

\begin{lemma}
Suppose that there exists a diffeomorphism $\Phi:(V_1\cup_{\tilde{f}}V_2)^*\to V_2^*\cup_{\tilde{f}^*}V_1^*$ such that $(\pi_1\cup_{(\tilde{f},f)}\pi_2)^*=\left(\pi_2^*\cup_{(\tilde{f}^*,f)}\pi_1^*\right)\circ\Phi$.
Then for every $y\in Y$ we have $(\pi_2^{-1}(f(y)))^*=(\pi_1^{-1}(y))^*$.
\end{lemma}

\begin{proof}
Let $y\in Y$ be a point in the domain of the gluing of the bases. Obviously, $y=f^{-1}(f(y))\sim f(y)$ represents a point both in $X_1\cup_f X_2$ and $X_2\cup_{f^{-1}}X_1$ (the latter two spaces being obviously 
diffeomorphic, but considered here as the base spaces of our two pseudo-bundles $(\pi_1\cup_{(\tilde{f},f)}\pi_2)^*$ and $\pi_2^*\cup_{(\tilde{f}^*,f)}\pi_1^*$ respectively). Now, observe that when we carry 
out a gluing, on each fibre the result is the target space of the gluing map.

Since in the first case the gluing over $y$ is along $\tilde{f}|:\pi_1^{-1}(y)\to\pi_2^{-1}(y)$ (with taking the dual afterwards), in the pseudo-bundle $(\pi_1\cup_{(\tilde{f},f)}\pi_2)^*$ the fibre over $y$ is
$(\pi_2^{-1}(y))^*$. On the other hand, $\tilde{f}^*$ goes from $(\pi_2^{-1}(y))^*$ to $(\pi_1^{-1}(y))^*$, so in the pseudo-bundle $\pi_2^*\cup_{(\tilde{f}^*,f)}\pi_1^*$ the fibre over $y$ is $(\pi_1^{-1}(y))^*$. Thus, 
we conclude that for $(\pi_1\cup_{(\tilde{f},f)}\pi_2)^*$ and $\pi_2^*\cup_{(\tilde{f}^*,f)}\pi_1^*$ to be diffeomorphic, it is necessary that for every $y\in Y$ we have $(\pi_2^{-1}(f(y)))^*=(\pi_1^{-1}(y))^*$.
\end{proof}

We will now show that the above condition is not sufficient, as can be expected from the standard case (that of vector bundles over smooth manifolds).

\subsubsection{An instance when the gluing-dual commutativity is absent: the open annulus and the open M\"obius strip} 

We now point out two simple examples, which are pseudo-bundles over the same base space ($S^1$ with its standard diffeology as a smooth manifold) and the same fibre (the standard $\matR$), but which 
are not diffeomorphic, nor even homeomorphic. These are the open annulus and the open M\"obius strip, both of which fibre over $S^1$ with fibre an open interval (which we identify with $\matR$).

Let us first specify what we mean by the standard diffeology on $S^1$. We can define it by representing it as the unit circle in $\matR^2$ and endow it with the subset diffeology. This diffeology we can also
define as the gluing diffeology on $S^1=(\matR_x\sqcup\matR_y)/_{x=\frac{1}{y}}$, where the gluing is along the set $\matR_x\setminus\{0\}$ and the gluing map $f$ on it is given by $f(x)=\frac{1}{x}$. 
The open annulus $\mathcal{A}$ carries then the product diffeology of $\mathcal{A}=S^1\times\matR$ of the direct product of $S^1$ with the standard $\matR$.

We now obtain the diffeology on the open M\"obius strip $\mathcal{M}$ by representing it as the total space of the pseudo-bundle obtained by gluing together the following two pseudo-bundles: 
$\pi_1:\matR^2_{(x_1,x_2)}\to\matR_{x_1}$ (that is, $\pi_1(x_1,x_2)=x_1$) and $\pi_2:\matR^2_{(y_1,y_2)}\to\matR_{y_1}$ with $\pi_2(y_1,y_2)=y_1$ (obviously, these are two copies of the same bundle). 
The domain of gluing is the set given by $x_1\neq 0$, and the map $f$ acts by $f(x_1)=\frac{1}{y_1}$. Finally, the lift $\tilde{f}$ is defined by setting $\tilde{f}(x_1,x_2)=(\frac{1}{x_1},x_2)$ for $x_1<0$ and
$\tilde{f}(x_1,x_2)=(\frac{1}{x_1},-x_2)$ for $x_1>0$.

Let us explain now how these two spaces relate to the question of $(\pi_2^{-1}(f(y)))^*=(\pi_1^{-1}(y))^*$ pointwise not being a sufficient condition for commutativity. The base spaces, \emph{i.e.}, the core 
circles of $\mathcal{A}$ and $\mathcal{M}$, play the role of $Y$ and its diffeomorphic image $f(Y)$; the total spaces, \emph{i.e.}, $\mathcal{A}$ and $\mathcal{M}$ themselves, act as $\pi_1^{-1}(Y)$ and 
$\pi_2^{-1}(f(Y))$, and since they are self-dual (we give the proof below), they also act as $(\pi_1^{-1}(Y))^*$ and $(\pi_2^{-1}(f(Y)))^*$. Note that if they are endowed with their standard diffeologies, any 
smooth lift $\tilde{f}:\mathcal{A}\to\mathcal{M}$ of any diffeomorphism between their core circles degenerates (becomes the zero map) on at least one fibre.

\begin{lemma}
Both $\mathcal{A}$ and $\mathcal{M}$, considered with their standard diffeologies, are self-dual as diffeological vector (pseudo-)bundles.
\end{lemma}

\begin{proof}
In the case of the annulus this simply follows from Theorem 5.3.5 of \cite{vincent}, which affirms that the dual of a trivial diffeological bundle is itself trivial. Therefore $\mathcal{A}^*$ is the trivial bundle with 
the same base and the same fibre as $\mathcal{A}$, and the conclusion is obvious. In this case of the M\"obius strip $\mathcal{M}$, the same uniqueness consideration would give the desired result if we 
show that the dual $\mathcal{M}^*$ is a non-trivial bundle. This follows from Theorem \ref{dual:diffeo:equals:dual:standard:thm} below which affirms that the dual in the diffeological sense bundle of a usual 
smooth vector bundle (under some assumptions that both the annulus and the M\"obius strip satisfy) is the same as its dual bundle in the usual sense (they are diffeomorphic if both are considered as
diffeological vector pseudo-bundles for the respective standard diffeologies).
\end{proof}

\section{Diffeological gluing and standard vector bundles}

The aim of this section, which does not contain any new results in the true sense, is to highlight the fact that the diffeological gluing is indeed an extension of the concept of a local atlas for smooth manifolds 
(when the latter are viewed with their standard diffeologies), and this extends all the way through to the operation of taking the dual bundle. More precisely, let $M^n$ be a smooth compact manifold with an 
atlas $\{(U_{\alpha},\varphi_{\alpha})\}_{\alpha\in I}$, where $U_{\alpha}\subset M$ is a finite covering and $\varphi_{\alpha}:U_{\alpha}\to\matR^n$. Then $M$, and each $U_{\alpha}$ can be seen as 
diffeological spaces, for their standard diffeological structures, and furthermore, $M$ can also be seen as the result of finite number of gluings of the spaces $U_{\alpha}$ along all the transition functions, and 
thus endowed with the corresponding gluing diffeology. The latter is \emph{a priori} finer than the standard smooth structure; the question is whether it can be strictly finer. Analogous question can then be 
asked for a usual vector bundle over $M$; and this is used to compare the usual dual bundle with its diffeological counterpart, showing in the end that these two viewpoints are equivalent.

\subsection{A usual smooth vector bundle via a local atlas vs. diffeological gluing}

Since the present section is for illustrative purposes mainly, we only give details as necessary to highlight the fact that these two presentations are actually the same. Consider $M=U_{\alpha}\cup U_{\beta}$, 
with $\varphi_{\alpha}:U_{\alpha}\to\matR^n$ and $\varphi_{\beta}:U_{\alpha}\to\matR^n$ being the charts defining the smooth structure of $M$. Let $Y=\varphi_{\alpha}(U_{\alpha}\cap U_{\beta})$, and let 
$g_{\alpha\beta}:Y\to\varphi_{\beta}(U_{\alpha}\cap U_{\beta})$ be the corresponding transition function. Let $\calD_s$ be the standard diffeology on $M$, namely such that a map $f:U\to M$, for $U$ a domain 
in $\matR^m$ is a plot for $\calD_s$ if and only if $f$ is smooth in the usual sense. The diffeological gluing allows to see $M$ as $(U_{\alpha}\sqcup U_{\beta})/\sim$, where $\sim$ is the equivalence relation 
given (essentially) by $x\sim g_{\alpha\beta}(x)$. Let $\calD_{gl}$ be the corresponding gluing diffeology.

\begin{lemma}\label{gluing:equals:standard:case2:lem}
The diffeologies $\calD_s$ and $\calD_{gl}$ coincide.
\end{lemma}

\begin{proof}
We need to show that each plot of $\calD_s$ is a plot of $\calD_{gl}$, and \emph{vice versa}. It suffices to observe that every $f:U\to M$ such that $\varphi_{\alpha}\circ f$ and 
$\varphi_{\beta}\circ f:U\to\matR^n$ are smooth, is locally a smooth function into either $U_{\alpha}$ or $U_{\beta}$, for the standard diffeology on each of them. Now, since $U_{\alpha},U_{\beta}$ are
open sets, so is their intersection, which, together with the properties of usual smooth maps, implies that for every point $u\in U$ there is an open neighbourhood $U'\ni u$ such that $f(U')$ is entirely 
contained in either $U_{\alpha}$ or $U_{\beta}$. The restriction of $f$ to $U'$ is then indeed a smooth map, as desired, so $f$ is a plot of $\calD_{gl}$.\footnote{Note that $U_{\alpha}\cap U_{\beta}$ being 
open is essential here; in general, the reasoning would fail if it is not.} This allows us to conclude that $\calD_s\subseteq\calD_{gl}$.

Let us now prove the \emph{vice versa}. Observe that, since the gluing is along a diffeomorphism, and $U_{\alpha}$ and $U_{\beta}$ are open, up to switch between $\alpha$ and $\beta$, we can locally
describe each plot of the gluing diffeology as a smooth map into $U_{\alpha}$ or as a composition $g_{\alpha\beta}\circ f$ for some smooth $f$ into $U_{\alpha}\cap U_{\beta}$, which implies that $p$
is smooth in the usual sense as a map into $M$, and so is a plot of $\calD_s$. This gives the reverse inclusion $\calD_{gl}\subseteq\calD_s$, and the statement is proven.
\end{proof}

The reasoning in the above lemma easily extends to the case of a usual vector bundle over a smooth connected manifold relative to a finite atlas of local trivializations. 

\begin{lemma}\label{gluing:equals:standard:bundle:case2:lem}
Let $\pi:E\to M$ be a smooth vector bundle of rank $k$ that admits an atlas of local trivializations consisting of two charts only, \emph{i.e.}, $M=U_1\cup U_2$ with $\varphi_i:U_i\to\matR^n$ a diffeomorphism 
for each $i=1,2$, and $E=\pi^{-1}(U_1)\cup\pi^{-1}(U_2)$ with $\psi_i:\pi^{-1}(U_i)\to U_i\times\matR^k$ being a fibrewise diffeomorphism for $i=1,2$. Let $g_{12}:U_1\cap U_2\to\mbox{GL}_k(\matR)$ be the 
corresponding transition function. Let $\pi_i:\matR^{n+k}\to\matR^n$ be the projection of $\matR^{n+k}=\matR^n\times\matR^k$ onto its first factor $\matR^n$ given by
$$\pi_i=\varphi_i\circ\pi\circ\psi_i^{-1}\circ(\varphi_i^{-1}\times\mbox{Id}_{\matR^k})$$ and considered as a diffeological vector pseudo-bundle with respect to the standard diffeology. Let 
$Y=\varphi_1(U_1\cap U_2)$, and let $f:Y\to\matR^n$ be given by $f=\varphi_2\circ(\varphi_1|_{U_1\cap U_2})^{-1}$. Finally, let $\tilde{f}:\pi_1^{-1}(Y)\to\pi_2^{-1}(f(Y))$ act by
$\tilde{f}(y,v)=(f(y),g_{12}(\varphi_1^{-1}(y))v)$. Then $\pi_1\cup_{(\tilde{f},f)}\pi_2:\matR^{n+k}\cup_{\tilde{f}}\matR^{n+k}\to\matR^n\cup_f\matR^n$ is diffeomorphic as a diffeological vector pseudo-bundle to 
the bundle $\pi:E\to M$.
\end{lemma}

\begin{proof}
Observe first of all that $(\tilde{f},f)$ does define a diffeological gluing of pseudo-bundles; this is an easy consequence of the definitions of the two maps, and of the relevant assumptions. Furthermore, the 
existence of the obvious diffeomorphism between $\matR^n\cup_f\matR^n$ and $M$ follows from Lemma \ref{gluing:equals:standard:case2:lem}. The same Lemma can also be applied to the gluing along 
$\tilde{f}$ between the two copies of $\matR^{n+k}$, allowing us to conclude that $E$ is diffeomorphic to $\matR^{n+k}\cup_{\tilde{f}}\matR^{n+k}$ (recall indeed that in the definition of gluing of two 
diffeological pseudo-bundles the gluing of the total spaces is still the usual gluing of them as diffeological spaces, and, on the other hand,
$\{(\pi^{-1}(U_1),(\varphi_1\times\mbox{Id}_{\matR^k})\circ\psi_1),(\pi^{-1}(U_2),(\varphi_2\times\mbox{Id}_{\matR^k})\circ\psi_2)\}$ is trivially a finite atlas on $E$). Finally, $\pi_1\cup_{(\tilde{f},f)}\pi_2$ and 
$\pi$ commute with these two diffeomorphisms by construction, so the statement is proven.
\end{proof}

A more general statement can then be easily obtained.

\begin{thm}\label{smooth:bundle:equals:diffeol:bundle:thm}
Let $\pi:E\to M$ be a smooth vector bundle of rank $k$ over an $n$-dimensional manifold $M$ that admits a finite atlas of $m$ local trivializations. Then $\pi$ is the result of gluing of $m$ diffeological 
vector pseudo-bundles $\pi_i:\matR^{n+k}\to\matR^n$, where
$$\pi_i=\varphi_i\circ\pi\circ\psi_i^{-1}\circ(\varphi_i^{-1}\times\mbox{Id}_{\matR^k})\mbox{ for }i=1,\ldots,m,$$ and the gluing $(\tilde{f}_{ij},f_{ij})$ between $\pi_i$ and $\pi_j$ is given by the maps
$$f_{ij}=\varphi_j\circ(\varphi_i|_{U_i\cap U_j})^{-1},\mbox{ and}$$ 
$$\tilde{f}_{ij}(y,v)=(f_{ij}(y),g_{ij}(\varphi_i^{-1}(y))v),\mbox{ for }y\in U_i\cap U_j\mbox{ and }v\in\pi_i^{-1}(y).$$
\end{thm}

\subsection{The usual dual and the diffeological dual}

Since duality works differently in diffeology with respect to the standard case, and its behavior with respect to the gluing is not straightforward, we are particularly interested in comparing the diffeological dual 
with the standard one. Let $\pi:E\to M$ be a smooth vector bundle of rank $k$ over an $n$-dimensional smooth manifold $M$, admitting a finite atlas of local trivializations. It is automatically a diffeological 
vector pseudo-bundle for the standard diffeologies on $E$ and $M$. In this section we use the diffeological gluing representation of the smooth bundle $\pi$ given by Theorem 
\ref{smooth:bundle:equals:diffeol:bundle:thm}, as well as our earlier observations regarding the behavior of gluing with respect to vector (pseudo-)bundles operations, to show that the usual dual bundle of 
$\pi:E\to M$ coincides with its diffeological dual pseudo-bundle (meaning that there exists a diffeological diffeomorphism between them).

\paragraph{Diffeological dual bundles of local trivializations} Let us consider one of the elementary bundles $\pi_i:\matR^{n+k}\to\matR^n$. If we consider it as a diffeological vector pseudo-bundle, with 
respect to the standard diffeologies on $\matR^{n+k}$ and $\matR^n$, by Theorem 5.3.5 of \cite{vincent} its diffeological dual is a trivial bundle with the fibre $(\matR^k)^*$; in particular, it does coincide with 
its usual dual.

\paragraph{Commutativity of gluing with dual} Let us now consider the question of whether the gluing of two such elementary bundles, say $\pi_i$ and $\pi_j$, along the maps $f_{ij}$ and $\tilde{f}_{ij}$ (see 
Theorem \ref{smooth:bundle:equals:diffeol:bundle:thm} for the definitions of these two maps), commutes with taking duals, that is, whether the pseudo-bundles
$$(\pi_i\cup_{(\tilde{f}_{ij},f_{ij})}\pi_j)^*:(\matR^{n+k}\cup_{\tilde{f}_{ij}}\matR^{n+k})^*\to\matR^n\cup_{f_{ij}}\matR^n,\mbox{ and}$$
$$\pi_j^*\cup_{(\tilde{f}_{ij}^*,f_{ij}^{-1})}\pi_i^*:(\matR^{n+k})^*\cup_{\tilde{f}_{ij}^*}(\matR^{n+k})^*\to\matR^n\cup_{f_{ij}^{-1}}\matR^n$$ are diffeomorphic. Since we are dealing solely with diffeomorphisms
and standard diffeologies, we certainly expect them to be so, but let us give a formal proof.

To prove our claim, let us denote by $i_1:\matR^n\to\matR^n\cup_{f_{ij}}\matR^n$ the inclusion of the first copy of $\matR^n$ into the result of gluing, and by $i_2:\matR^n\to\matR^n\cup_{f_{ij}}\matR^n$ the 
inclusion of the second copy. Let $j_1$ and $j_2$ be the analogously defined inclusions for the space $\matR^n\cup_{f_{ij}^{-1}}\matR^n$. Note that all four maps are actually \emph{inductions} (\emph{i.e.}, 
each of them is a diffeomorphism with its image). Besides, there is an obvious (and unique) diffeomorphism $\varphi:\matR^n\cup_{f_{ij}}\matR^n\to\matR^n\cup_{f_{ij}^{-1}}\matR^n$, with the property that 
$\varphi\circ i_1=\varphi^{-1}\circ j_2$ and \emph{vice versa} (that is, $\varphi\circ i_2=\varphi^{-1}\circ j_1$). We call this $\varphi$ the \textbf{switch map} (since it does indeed exchange the copies of 
$\matR^n$ being glued together).

\begin{prop}\label{diffeological:dual:is:standard:dual:prop}
Let $\varphi:\matR^n\cup_{f_{ij}}\matR^n\to\matR^n\cup_{f_{ij}^{-1}}\matR^n$ be the switch map between the two copies of $\matR^n$. Then there exists a lift
$\Phi:(\matR^{n+k}\cup_{\tilde{f}_{ij}}\matR^{n+k})^*\to(\matR^{n+k})^*\cup_{\tilde{f}_{ij}^*}(\matR^{n+k})^*$ which is a diffeomorphism.
\end{prop}

\begin{proof}
The definition of $\Phi$ is rather obvious; examining the fibres over any arbitrary $x\in\matR^n\cup_{f_{ij}}\matR^n$ and $\varphi(x)\in\matR^n\cup_{f_{ij}^{-1}}\matR^n$, we notice that they are diffeomorphic 
(essentially, they are identical). Hence so are their duals (individually), which gives an obvious bijection, $(\matR^{n+k}\cup_{\tilde{f}_{ij}}\matR^{n+k})^*\to(\matR^{n+k})^*\cup_{\tilde{f}_{ij}^*}(\matR^{n+k})^*$,
that uses the canonical isomorphism between $\matR^k$ and its dual space and that we denote $\Phi$. It remains to see that $\Phi$ is smooth with smooth inverse, so we actually spell out its (local) form 
in detail.

Let us denote for brevity $V_{\cup,*}:=(\matR^{n+k}\cup_{\tilde{f}_{ij}}\matR^{n+k})^*$ and $V_{*,\cup}:=(\matR^{n+k})^*\cup_{\tilde{f}_{ij}^*}(\matR^{n+k})^*$. Since $\tilde{f}_{ij}^*$ is also a diffeomorphism, 
by Lemma \ref{gluing:equals:standard:case2:lem} the gluing diffeology on $V_{*,\cup}$ coincides with the standard one, while, using the same Lemma, the diffeology on $V_{\cup,*}$ is the dual (in the sense 
of pseudo-bundles for now) diffeology of the standard diffeology on $\matR^{n+k}\cup_{\tilde{f}_{ij}}\matR^{n+k}$. A plot $p:U\to V_{\cup,*}$ of the diffeology on $V_{\cup,*}$ is characterized by the property 
that the evaluation $p(u)(f(u'))$ of $p$ on any smooth function $f:U'\to\matR^{n+k}$ (due to the fact that the two copies of $\matR^{n+k}$ are being glued along the open subsets 
$\varphi(U_i\cap U_j)\times\matR^k$ and $\varphi_j(U_i\cap U_j)\times\matR^k$, the local form of any plot of $\matR^{n+k}\cup_{\tilde{f}_{ij}}\matR^{n+k}$ is such an $f$) is smooth for the subset diffeology 
(relative to the standard diffeology on the domain $U\times U'$) on the set of pairs $(u,u')$ such that $(\pi_i\cup_{(\tilde{f}_{ij},f_{ij})}\pi_j)^*(u)=(\pi_i\cup_{(\tilde{f}_{ij},f_{ij})}\pi_j)(u')$. This implies that we 
can always choose $U$ small enough so that $(\pi_i\cup_{(\tilde{f}_{ij},f_{ij})}\pi_j)(p(U))$ be wholly contained in one of the copies of $\matR^n\subset\matR^n\cup_{f_{ij}}\matR^n$.

Furthermore, this implies that the plot $p$ can be identified with either a plot $p_1$ of the pre-image $((\pi_i\cup_{(\tilde{f}_{ij},f_{ij})}\pi_j)^*)^{-1}(\matR^n)\subset V_{\cup,*}$ of the first copy of $\matR^n$, or a 
plot $p_2$ of the pre-image of the second copy of $\matR^n$. Thus, $p(u)(f(u'))$ coincides with either $p_1(u)(f(u'))$ or $p_2(u)(f(u'))$. Furthermore, on $\varphi_i(U_i\cap U_j)$ (the domain of definition of 
$f_{ij}$) we have $p_1(u)=p_2(u)\circ f_{ij}$, \emph{i.e.} $p_1(u)=f_{ij}^*(p_2(u))$. Finally, observe that $\Phi\circ p$ is either $p_1$ or $p_2$ considered as element(s) of $V_{*,\cup}$, so a plot of the latter. 
Since smoothness is a local property, we obtain the $\Phi$ is indeed a smooth map. Finally, the smoothness of the inverse map $\Phi^{-1}$ is established in exactly the same way, so we conclude that $\Phi$ 
is a diffeomorphism, as wanted.
\end{proof}

We emphasize that the main reason why this statement, that is not true in general, holds in the present case is that the gluing is performed along a diffeomorphism between two \emph{open} subsets
(both in the bases and in the total spaces).

\paragraph{The general statement} We can now use the representation of the bundle $\pi:E\to M$ as the result of gluing of the elementary standard bundles $\pi_i:\matR^{n+k}\to\matR^n$ given by Theorem
\ref{smooth:bundle:equals:diffeol:bundle:thm}, to obtain the following claim.

\begin{thm}\label{dual:diffeo:equals:dual:standard:thm}
Let $\pi:E\to M$ be a smooth vector bundle over a smooth finite-dimensional manifold $M$, that admits a finite atlas of local trivializations. Then the diffeological dual pseudo-bundle of $\pi:E\to M$ endowed 
with the standard diffeological structure is diffeomorphic to its usual dual bundle, via the identity on $M$.
\end{thm}

\section{Smooth maps and gluing}

The following somewhat technical construction is frequently used when dealing with the gluing of diffeological spaces. It has to do with the situation where we glue together two diffeological spaces each of 
which acts as the domain of some smooth map. The range of the two maps could be the same diffeological space, or else there could be a simultaneously gluing of the ranges (such as when we are
dealing with the gluing of the pseudo-bundles and two respective sections of them). Now, under an expected condition of compatibility there is an induced map on the result, which can be called the
result of gluing of the two maps (on of which the pseudo-bundle map of form $\pi_1\cup_{(\tilde{f},f)}\pi_2$ is a specific instance).

\subsection{Local structure of plots of gluing diffeology}\label{natural:inductions:sect}

When it comes to considering the smoothness of various maps on diffeological spaces obtained by gluing, it is useful to consider the following technical properties of the gluing diffeology.

\begin{lemma}\label{describe:plots:gluing:diffeol:lem}
Let $X_1$ and $X_2$ be two diffeological spaces, and let $f:X_1\supset Y\to X_2$ be a smooth map. Then:\\
1) the obvious inclusions $i_1:X_1\setminus Y\hookrightarrow X_1\cup_f X_2$, where $X_1\setminus Y$ is considered with the subset diffeology relative to the diffeology of $X_1$, and
$i_2:X_2\hookrightarrow X_1\cup_f X_2$ are inductions;\footnote{A map $f:X\to Y$ between two diffeological spaces is called an induction if it is injective, and the diffeology of $X$ is the pullback, via $f$, of 
the diffeology of $Y$.}\\
2) every plot $p:U\to X_1\cup_f X_2$ locally has the following characterization: either there exists a plot $p_2:U\supset U'\to X_2$ of $X_2$ such that $p|_{U'}=i_2\circ p_2$, or there exists a plot 
$p_1:U\supset U'\to X_1$ of $X_1$ such that 
$$p(u')=\left\{\begin{array}{ll}i_1(p_1(u')) & \mbox{if }p_1(u')\in X_1\setminus Y,\\ i_2(f(p_1(u'))) & \mbox{if }p_1(u')\in Y.\end{array}\right.$$
\end{lemma}

\begin{proof}
This lemma is essentially a direct consequence of the definitions. The only somewhat subtle aspect to clarify is why $i_1$ and $i_2$ are inductions and not just smooth injective maps. We need to check
that $p_1:U\to X_1\setminus Y$ is a plot of $X_1\setminus Y$ if and only if $i_1\circ p_1$ is a plot of $X_1\cup_f X_2$, as well as the analogous statement for $i_2$. Recall first of all that the diffeology on 
$X_1\cup_f X_2$ is a pushforward of the diffeology on $X_1\sqcup X_2$, therefore $i_1\circ p_1$ is a plot of $X_1\cup_f X_2$ if and only if it locally lifts to a plot of $X_1\sqcup X_2$; this lift is unique in this 
case, so $i_1\circ p_1$ is a plot if and only if the composition of the obvious inclusions $(X_1\setminus Y)\hookrightarrow X_1\hookrightarrow (X_1\sqcup X_2)$ is a plot of $X_1\sqcup X_2$. The second of 
these inclusions is an induction by the definition of the disjoint union diffeology, and so is the first one, by definition of the subset diffeology. We conclude observing that the composition of two inductions is an 
induction itself.

Likewise, let $p_2:U\to X_2$ be a map. The composition $i_2\circ p_2$ is a plot of $X_1\cup_f X_2$ if and only if it lifts to a plot of $X_1\sqcup X_2$. Once again, the lift of this map is unique and coincides 
the inclusion $X_2\hookrightarrow(X_1\sqcup X_2)$, which is an induction. Thus, $i_2\circ p_2$ is a plot if and only if $p_2$ is a plot.
\end{proof}

\subsection{The compatibility condition}\label{f-g:compatibility:sect}

This condition is rather obvious; to state it rigorously, we need to distinguish between the following two cases.

\paragraph{The gluing of two $Z$-valued maps} Let $X_1,X_2,Z$ be diffeological spaces, let $f:X_1\supset Y\to X_2$ be a smooth map for the subset diffeology on $Y$, and let $\varphi_i:X_i\to Z$ be
two other smooth maps, with $i=1,2$. We say that $\varphi_1$ and $\varphi_2$ are \textbf{$f$-compatible} (or simply \textbf{compatible} when the map of reference is clear from the context) if the following is 
true for all $y\in Y$:
$$\varphi_1(y)=\varphi_2(f(y)).$$ Recall that this equality makes sense because both $\varphi_1$ and $\varphi_2$ take values in the same set $Z$ (which for this notion does not actually need to be 
diffeological space; it could be any set).

\paragraph{The gluing of two maps with different ranges} Let $X_1,X_2$ and $Z_1,Z_2$ be two pairs of diffeological spaces, and let $\psi_1:X_1\to Z_1$ and $\psi_2:X_2\to Z_2$ be two smooth maps.
Suppose, further, that we are given a smooth map $f:X_1\supset Y\to X_2$ and a smooth map $g:Y'\to Z_2$ defined on $Y'=\psi_1(Y)\subset Z_1$. We say that $\psi_1$ and $\psi_2$ are
\textbf{$(f,g)$-compatible} if for all $y\in Y$ we have  $$g(\psi_1(y))=\psi_2(f(y)).$$

\subsection{The gluing of $\varphi_i:X_i\to Z$ along $f:X_1\supset Y\to X_2$}\label{f-gluing:of:maps:sect}

Let us now describe the map on the glued space $X_1\cup_f X_2$ induced by two given $f$-compatible maps $\varphi_1:X_1\to Z$ and $\varphi_2:X_2\to Z$. The definition of this induced map, which we
denote by $\varphi_1\cup_f\varphi_2$ is the evident one. It is a map $\varphi_1\cup_f\varphi_2:X_1\cup_f X_2\to Z$ that acts by
$$(\varphi_1\cup_f\varphi_2)(x)=\left\{\begin{array}{ll} \varphi_1(i_1^{-1}(x)) & \mbox{for }x\in i_1(X_1\setminus Y)\\ \varphi_2(i_2^{-1}(x)) & \mbox{for }x\in i_2(X_2).\end{array}\right.$$

Since $i_1(X_1\setminus Y)$ and $i_2(X_2)$ cover $X_1\cup_f X_2$ and are disjoint (see Lemma \ref{describe:plots:gluing:diffeol:lem}), this map is well-defined. In fact, it is so even without the requirement 
of compatibility, which, on the other hand, is necessary to establish the following statement.

\begin{prop}
The map $\varphi_1\cup_f\varphi_2$ is smooth for the gluing diffeology on $X_1\cup_f X_2$ and the given diffeology of $Z$.
\end{prop}

\begin{proof}
Let $p:U\to X_1\cup_f X_2$ be a plot of $X_1\cup_f X_2$; we need to show that $(\varphi_1\cup_f\varphi_2)\circ p$ is a plot of $Z$. By Lemma \ref{describe:plots:gluing:diffeol:lem}, for every point 
$u\in U$ there is a small neighbourhood $U'\ni u$ such that the restriction of $p$ to $U'$ either has form $i_2\circ p_2$ for some plot $p_2:U'\to X_2$, or else there exists a plot $p_1$ of $X_1$ such that
$$p(u')=\left\{\begin{array}{ll} (i_1\circ p_1)(u') & \mbox{if }p_1(u')\notin Y,\mbox{ and} \\ (i_2\circ f\circ p_1)(u') & \mbox{if }p_1(u')\in Y, \end{array}\right.$$ for all $u'\in U'$. In the first case we have:
$$((\varphi_1\cup_f\varphi_2)\circ p)(u')=(\varphi_1\cup_f\varphi_2)(i_2(p_2(u')))=(\varphi_2\circ p_2)(u')\mbox{ for all }u'\in U'.$$ Since $\varphi_2$ is smooth by assumption, and $p_2$ is a plot of $X_2$,
the composition $\varphi_2\circ p_2$ is a plot of $Z$, as we wanted.

Let us now consider the second case. We have:
$$((\varphi_1\cup_f\varphi_2)\circ p)(u')=\left\{\begin{array}{ll} \varphi_1(p_1(u')) & \mbox{if }p_1(u')\notin Y,\mbox{ and} \\ \varphi_2(f(p_1(u')))=\varphi_1(p_1(u')) & \mbox{if }p_1(u')\in Y, \end{array}\right.$$
where the equality $\varphi_2(f(p_1(u')))=\varphi_1(p_1(u'))$ follows from the compatibility condition. Thus, we obtain that $((\varphi_1\cup_f\varphi_2)\circ p)(u')=(\varphi_1\circ p_1)(u')$ for all $u'\in U'$; 
since $\varphi_1$ is smooth by assumption and $p_1$ is a plot of $X_1$, we conclude that $(\varphi_1\cup_f\varphi_2)\circ p=\varphi_1\circ p_1$ is also a plot of $Z$, so we obtain the final statement.
\end{proof}

We thus have an assignment $(\varphi_1,\varphi_2)\mapsto\varphi_1\cup_f\varphi_2$, defined for every pair of $f$-compatible maps $\varphi_i:X_i\to Z$, and yielding a smooth map 
$\varphi_1\cup_f\varphi_2:X_1\cup_f X_2\to Z$. This assignment can be seen as a map 
$$\mathcal{F}_Z:C^{\infty}(X_1,Z)\times_{comp}C^{\infty}(X_2,Z)\to C^{\infty}(X_1\cup_f X_2,Z),$$ defined on the subset of the direct product $C^{\infty}(X_1,Z)\times C^{\infty}(X_2,Z)$ that consists of all pairs 
of $f$-compatible maps. Now, each of $C^{\infty}(X_1,Z)$, $C^{\infty}(X_2,Z)$, and $C^{\infty}(X_1\cup_f X_2,Z)$ has its standard functional diffeology, and the subset 
$C^{\infty}(X_1,Z)\times_{comp}C^{\infty}(X_2,Z)$ has the natural subset diffeology coming from the product diffeology on $C^{\infty}(X_1,Z)\times C^{\infty}(X_2,Z)$. With respect to all of these, we have the 
following statement.

\begin{thm}\label{f-compatible:gluing:smooth:thm}
The map $\mathcal{F}_Z:C^{\infty}(X_1,Z)\times_{comp}C^{\infty}(X_2,Z)\to C^{\infty}(X_1\cup_f X_2,Z)$ is smooth.
\end{thm}

\begin{proof}
Let $p:U\to C^{\infty}(X_1,Z)\times_{comp}C^{\infty}(X_2,Z)$ be a plot; by definition of the subset diffeology, $p$ is then a plot of $C^{\infty}(X_1,Z)\times C^{\infty}(X_2,Z)$ such that its range is contained in 
$C^{\infty}(X_1,Z)\times_{comp}C^{\infty}(X_2,Z)$. By definition of the direct product diffeology, $p$ has form $p=(p_1,p_2)$, where each $p_i$ is a plot of $C^{\infty}(X_i,Z)$ for $i=1,2$.

Recall that, by the standard property of any functional diffeology, $p_i$ being a plot of $C^{\infty}(X_i,Z)$ means that for every plot $q_i:U'\to X_i$ of $X_i$ and $i=1,2$ the map $(u,u')\mapsto p_i(u)(q_i(u'))$ 
is a plot of $Z$. By the same property, to show that $\mathcal{F}_Z$ is smooth, we need to show that $\mathcal{F}_Z\circ p$ is a plot of $C^{\infty}(X_1\cup_f X_2,Z)$, and specifically, that for every plot 
$q:U'\to X_1\cup_f X_2$ of $X_1\cup_f X_2$ the induced map $(u,u')\mapsto(\mathcal{F}_Z\circ p)(u)(q(u'))$ is a plot of $Z$.

Let $q$ be an arbitrary plot of $X_1\cup_f X_2$; let $u$ be a point in the domain of the definition of $q$, and let $U'$ be a small enough neighborhood of $u$ in this domain such that $q$ lifts to either a plot 
$q_1$ of $X_1$ or a plot $q_2$ of $X_2$. Suppose first that $q$ lifts to a plot of $X_2$, \emph{i.e.}, $q=i_2\circ q_2$ on $U'$. Then by construction
$$(\mathcal{F}_Z\circ p)(u)(q(u'))=p_2(u)(q_2(u'))$$ on $U'$; therefore it is a plot of $Z$ by the assumptions on $p_2$ and $q_2$.

Suppose now that $q$ lifts to a plot $q_1$ of $X_1$. This means that on $U'$ the plot $q$ is defined as follows:
$$q(u')=\left\{\begin{array}{ll} i_1(q_1(u')) & \mbox{for }u'\in U' \mbox{ such that }q_1(u')\in X_1\setminus Y,\mbox{ and}\\
i_2(f(q_1(u'))) & \mbox{for }u'\in U'\mbox{ such that }q_1(u')\in Y.\end{array}\right.$$ In this case we have 
$$(\mathcal{F}_Z\circ p)(u)(q(u'))=\left\{\begin{array}{ll} p_1(u)(q_1(u')) & \mbox{for }u'\in U'\mbox{ such that }q_1(u')\in X_1\setminus Y,\mbox{ and}\\
p_2(u)(f(q_1(u')))=p_1(u)(q_1(u')) & \mbox{for }u'\in U'\mbox{ such that }q_1(u')\in Y,\end{array}\right.$$ where the equality in the second line is by the compatibility condition. This allows to conclude that on 
$U\times U'$ the map $\mathcal{F}_Z\circ p$ coincides with the evaluation of $p_1$ at $q_1$, which is a plot of $Z$ by the assumptions on $p_1$ and $q_1$, therefore the conclusion.
\end{proof}

\subsection{The gluing of $\psi_i:X_i\to Z_i$ along $f:X_1\supset Y\to X_2$ and $g:Z_1\supset Y'\to Z_2$}\label{f-g:gluing:of:maps:sect}

Let us now consider the case of a simultaneous gluing on the domains and the ranges of two given smooth maps $\psi_1:X_1\to Z_1$ and $\psi_2:X_2\to Z_2$ between diffeological spaces. On the domains 
the gluing is given by the map $f:X_1\supset Y\to X_2$, and on the ranges it is given by $g:Z_1\supset\psi_1(Y)\to Z_2$; these maps give rise to two new diffeological spaces, $X_1\cup_f X_2$ and
$Z_1\cup_g Z_2$. Denote, as before, by $i_1$ and $i_2$, for $k=1,2$, the inductions $i_1:X_1\setminus Y\hookrightarrow X_1\cup_f X_2$ and $i_2:X_2\hookrightarrow X_1\cup_f X_2$ (these are the same 
$i_1$ and $i_2$ that are defined in Section 4.1), and by $j_1$ and $j_2$ the analogous inductions for $Z_1$, $Z_2$, and $g$, that is, $j_1:Z_1\setminus\psi_1(Y)\hookrightarrow Z_1\cup_g Z_2$ and
$j_2:Z_2\hookrightarrow Z_1\cup_g Z_2$.

Assuming that $\psi_1$ and $\psi_2$ are $(f,g)$-compatible, we can define the following map $\psi_1\cup_{(f,g)}\psi_2:X_1\cup_f X_2\to Z_1\cup_g Z_2$ between these new spaces:\footnote{As before, the 
map itself can be defined without the assumption of the compatibility, which is needed to guarantee its smoothness.}
$$(\psi_1\cup_{(f,g)}\psi_2)(x)=\left\{\begin{array}{ll} j_1(\psi_1(i_1^{-1}(x))) & \mbox{for }x\in i_1(X_1\setminus Y)\\
j_2(\psi_2(i_2^{-1}(x))) & \mbox{for }x\in i_2(X_2).\end{array}\right.$$ As we see, the definition of the map $\psi_1\cup_{(f,g)}\psi_2$ is very much similar to that of the map $\varphi_1\cup_f\varphi_2$
(especially considering that our use of $j_1$ and $j_2$ serves mostly the formal purposes); the proof that it is smooth is a bit different.

\begin{prop}
The map $\psi_1\cup_{(f,g)}\psi_2$ is smooth for the gluing diffeologies on $X_1\cup_f X_2$ and $Z_1\cup_g Z_2$.
\end{prop}

\begin{proof}
Let $p:U\to X_1\cup_f X_2$ be a plot of $X_1\cup_f X_2$; we assume from the start that $U$ is small enough so that $p$ lifts to either a plot of $X_1$ or one of $X_2$. Suppose first that $p=i_2\circ p_2$
for an appropriate plot $p_2:U\to X_2$ of $X_2$. Then we have $(\psi_1\cup_{(f,g)}\psi_2)(p(u))=j_2(\psi_2(p_2(u)))$; since $\psi_2$ is smooth by assumption, $\psi_2\circ p_2$ is a plot of $Z_2$, and 
since $j_2$ is an induction, $j_2\circ\psi_2\circ p_2$ is a plot $Z_1\cup_g Z_2$, as wanted.

Suppose now that $p$ lifts to a plot $p_1:U\to X_1$ of $X_1$. Then we obtain
$$(\psi_1\cup_{(f,g)}\psi_2)(p(u))=\left\{\begin{array}{ll} j_1(\psi_1(p_1(u))) & \mbox{for }u\mbox{ such that }p_1(u)\in X_1\setminus Y\\
j_2(\psi_2(f(p_1(u)))) & \mbox{for }u\mbox{ such that }p_1(u)\in Y.\end{array}\right.$$ Now, observe that if $p_1(u)\in Y$ then $j_2(\psi_2(f(p_1(u))))=j_2(g(\psi_1(p_1(u))))$ by the compatibility condition, 
and $\psi_1\circ p_1$ is a plot of $Z_1$, since $\psi_1$ is smooth. Let us denote $q_1=\psi_1\circ p_1:U\to Z_1$ this plot; we obtain that
$$(\psi_1\cup_{(f,g)}\psi_2)(p(u))=\left\{\begin{array}{ll} j_1(q_1(u))) & \mbox{for }u\mbox{ such that }q_1(u)\in Z_1\setminus\psi_1(Y)\\
j_2(g(q_1(u))) & \mbox{for }u\mbox{ such that }q_1(u)\in\psi_1(Y).\end{array}\right.$$ It then follows from Lemma \ref{describe:plots:gluing:diffeol:lem} that this is a plot of $Z_1\cup_g Z_2$, which completes 
the proof.
\end{proof}

\begin{rem}
The proposition just proven will mostly be applied in the case when $\psi_1$ and $\psi_2$ are two sections of, respectively, pseudo-bundles $\pi_1:V_1\to X_1$ and $\pi_2:V_2\to X_2$ (so $Z_i=V_i$ and 
$g$ is $\tilde{f}$) and, especially, in the case of sections of the corresponding pseudo-bundles $\pi_i^*\otimes\pi_i^*:V_i^*\otimes V_i^*\to X_i$, which in what follows we call \emph{pseudo-metrics} on 
pseudo-bundles $\pi_i$. 
\end{rem}

Assigning to each pair of smooth $(f,g)$-compatible maps $\psi_1$ and $\psi_2$ the above-defined map $\psi_1\cup_{(f,g)}\psi_2$ yields a well-defined map
$$\mathcal{F}_{Z_1,Z_2}:C^{\infty}(X_1,Z_1)\times_{comp}C^{\infty}(X_2,Z_2)\to C^{\infty}(X_1\cup_f X_2,Z_1\cup_g Z_2).$$ Both the domain of definition and the range of $\mathcal{F}_{Z_1,Z_2}$ are 
naturally diffeological spaces: $C^{\infty}(X_1\cup_f X_2,Z_1\cup_g Z_2)$ has the standard functional diffeology, while $C^{\infty}(X_1,Z_1)\times_{comp}C^{\infty}(X_2,Z_2)$ carries the subset diffeology 
coming from the inclusion $C^{\infty}(X_1,Z_1)\times_{comp}C^{\infty}(X_2,Z_2)\subset C^{\infty}(X_1,Z_1)\times C^{\infty}(X_2,Z_2)$, with the latter space considered with the product diffeology relative to 
the functional diffeologies on $C^{\infty}(X_1,Z_1)$ and $C^{\infty}(X_2,Z_2)$. With respect to these diffeologies the following statement holds.

\begin{thm}\label{f-g-compatible:gluing:smooth:thm}
The map $\mathcal{F}_{Z_1,Z_2}$ is smooth.
\end{thm}

\begin{proof}
Let $\varphi_i:X_i\to Z_1\cup_g Z_2$ be the composition of $\psi_i$ with the following: $Z_i\hookrightarrow(Z_1\sqcup Z_2)\to Z_1\cup_g Z_2$, for $i=1,2$. As a composition of three smooth maps, each
$\varphi_i$ is a smooth map.

Let us show that $\varphi_1$ and $\varphi_2$ are $f$-compatible. We need to check that for all $y\in Y$ we have $\varphi_2(f(y))=\varphi_1(y)$. But by their definitions, $\varphi_1(y)=g(\psi_1(y))$, and 
$\varphi_2(f(y))=\psi_2(f(y))$, so the desired equality follows from the $(f,g)$-compatibility between $\psi_1$ and $\psi_2$.

Let us show that $\varphi_1\cup_f\varphi_2$ coincides with $\psi_1\cup_{(f,g)}\psi_2$. By definition of the former, 
$$(\varphi_1\cup_f\varphi_2)(x)=\left\{\begin{array}{ll}\varphi_1(i_1^{-1}(x))=j_1(\psi_1(i_1^{-1}(x))) & \mbox{for }x\in i_1(X_1\setminus Y),\\
\varphi_2(i_2^{-1}(x))=j_2(\psi_2(i_2^{-1}(x))) & \mbox{for }x\in i_2(X_2).\end{array}\right.$$ But the right-hand sides of the two equalities above define precisely the map $\psi_1\cup_{(f,g)}\psi_2$, 
as wanted.

Observe now that the assignment $\psi_i\mapsto\varphi_i$ yields a well-defined map $\Phi_i:C^{\infty}(X_i,Z_i)\to C^{\infty}(X_i,Z_1\cup_g Z_2)$; and since this assignment acts by a post-composition with 
the same fixed map, $\Phi_i$ is smooth for the functional diffeologies on the two spaces involved. Furthermore, the appropriate restriction $\Phi$ of
$(\Phi_1,\Phi_2):C^{\infty}(X_1,Z_1)\times C^{\infty}(X_2,Z_2)\to C^{\infty}(X_1,Z_1\cup_g Z_2)\times C^{\infty}(X_2,Z_1\cup_g Z_2)$ is a well-defined smooth map 
$C^{\infty}(X_1,Z_1)\times_{comp}C^{\infty}(X_2,Z_2)\to C^{\infty}(X_1,Z_1\cup_g Z_2)\times_{comp}C^{\infty}(X_2,Z_1\cup_g Z_2)$, where the former is with respect to the $(f,g)$-compatibility and the latter 
is with respect to the $f$-compatibility. It remains to observe that $\mathcal{F}_{Z_1,Z_2}=\mathcal{F}_{Z_1\cup_g Z_2}\circ\Phi$, so applying Theorem \ref{f-compatible:gluing:smooth:thm} for 
$Z=Z_1\cup_g Z_2$, we conclude that $\mathcal{F}_{Z_1,Z_2}$ is smooth, as desired.
\end{proof}

\subsection{Gluing together two sections of two pseudo-bundles}

Let $\pi_1:V_1\to X_1$ and $\pi_2:V_2\to X_2$ be two diffeological vector pseudo-bundles, and let $(\tilde{f},f)$ be a gluing between them. For two sections $s_1\in C^{\infty}(X_1,V_1)$ and 
$s_2\in C^{\infty}(X_2,V_2)$ being $(f,\tilde{f})$-compatible is just a partial case of the $(f,g)$-compatibility above, and if $s_1$ and $s_2$ are so, the image 
$\mathcal{F}_{V_1,V_2}(s_1,s_2)=s_1\cup_{(f,\tilde{f})}s_2$ is a smooth section of the pseudo-bundle $\pi_1\cup_{(\tilde{f},f)}\pi_2:V_1\cup_{\tilde{f}}V_2\to X_1\cup_f X_2$. Below we prove two statements 
applying specifically to this setting.

\begin{prop}
Let $s_1\in C^{\infty}(X_1,V_1)$ and $s_2\in C^{\infty}(X_2,V_2)$ be $(f,\tilde{f})$-compatible, and let $h_1\in C^{\infty}(X_1,\matR)$ and $h_2\in C^{\infty}(X_2,\matR)$ be $f$-compatible functions, that is, 
such that $h_1(y)=h_2(f(y))$ for all $y\in Y$. Then
$$(h_1\cup_f h_2)(s_1\cup_{(f,\tilde{f})}s_2)=(h_1s_1)\cup_{(f,\tilde{f})}(h_2s_2).$$
\end{prop}

\begin{proof}
The proof is by pointwise consideration of the sections on the two sides of the equality. Namely, recall that
$$(h_1\cup_f h_2)(x)=\left\{\begin{array}{ll} h_1(i_1^{-1}(x)) & \mbox{if }x\in\mbox{Im}(i_1),\\ h_2(i_2^{-1}(x)) & \mbox{if }x\in\mbox{Im}(i_2), \end{array}\right.$$ and accordingly,
$$(s_1\cup_{(f,\tilde{f})} s_2)(x)=\left\{\begin{array}{ll} j_1(s_1(i_1^{-1}(x))) & \mbox{if }x\in\mbox{Im}(i_1),\\ j_2(s_2(i_2^{-1}(x))) & \mbox{if }x\in\mbox{Im}(i_2).\end{array}\right.$$ Therefore
$$(h_1\cup_f h_2)(s_1\cup_{(f,\tilde{f})} s_2)(x)=\left\{\begin{array}{ll} (h_1(i_1^{-1}(x)))\cdot j_1(s_1(i_1^{-1}(x)))=j_1((h_1s_1)(i_1^{-1}(x))) & \mbox{if }x\in\mbox{Im}(i_1),\\
(h_2(i_2^{-1}(x)))\cdot j_2(s_2(i_2^{-1}(x)))=j_2((h_2s_2)(i_2^{-1}(x))) & \mbox{if }x\in\mbox{Im}(i_2),\end{array}\right.$$ whence the conclusion.
\end{proof}

The next statement essentially follows from the already-seen commutativity between gluing and tensor product (in particular, the proof is a combination of the pointwise consideration of the two sections, 
and the above-mentioned commutativity).

\begin{prop}
Let $s_i\in C^{\infty}(X_i,V_i)$ and $s_i'\in C^{\infty}(X_i,V_i')$ be such that $s_1,s_2$ and $s_1',s_2'$ are two pairs of $(f,\tilde{f})$-compatible sections. Let $\tilde{f}'$ be a lift of
$f$ to pseudo-bundles $V_1',V_2'$. Then
$$(s_1\cup_{(f,\tilde{f})}s_2)\otimes(s_1'\cup_{(f,\tilde{f}')}s_2')=(s_1\otimes s_1')\cup_{(f,\tilde{f}\otimes\tilde{f}')}(s_2\otimes s_2').$$
\end{prop}

\section{Pseudo-metrics on pseudo-bundles}

In this section we consider the notion obtained by a natural extension of the usual definition of a Riemannian metric to the case of diffeological vector pseudo-bundles. This extension has two aspects: one is 
that the requirement of smoothness assumes its diffeological meaning; the other is that instead of asking for a scalar product on each fibre, we ask for there to be a pseudo-metric in the sense of diffeological 
vector spaces (see \cite{pseudometric} and Section 1). The notion thus obtained is again called a (diffeological) \emph{pseudo-metric} on a pseudo-bundle; it was already considered in \cite{pseudobundles}. 
Here we build on and extend the results of the latter (briefly citing them as necessary).

\subsection{The definition and examples}

As has been indicated above, the definition of a pseudo-metric on a finite-dimensional diffeological vector pseudo-bundle is given in a form very much similar to that of a usual Riemannian metric.

\begin{defn}
Let $\pi:V\to X$ be a diffeological vector pseudo-bundle with finite-dimensional fibres. A \textbf{pseudo-metric} on this pseudo-bundle is a smooth section $g:X\to V^*\otimes V^*$ of the pseudo-bundle 
$V^*\otimes V^*$ such that for every $x\in X$ the bilinear form $g(x)$ is symmetric, positive semi-definite, and of rank equal to $\dim((\pi^{-1}(x))^*)$ (\emph{i.e.}, $g(x)$ is a pseudo-metric on the diffeological 
vector space $\pi^{-1}(x)$).
\end{defn}

Let us also briefly explain why $g(x)$ is, in fact, a bilinear form. This is for the same reason as it would be so in the usual Riemannian context: indeed, by Theorem 2.3.5 of \cite{vincent} there is a 
diffeomorphism between the space of bilinear forms on any diffeological vector space $W$ and $(W\otimes W)^*$, and it was shown in \cite{multilinear} that $(W\otimes W)^*$ is indeed diffeomorphic to 
$W^*\otimes W^*$.

A usual metric on the tangent bundle of a smooth manifold is of course an example of a pseudo-metric on a pseudo-bundle, but let us give another one, more specific to diffeology.

\begin{example}
Take $X$ to be the standard $\matR$, and endow $V=\matR^3$ with the pseudo-bundle diffeology, relative to its projection onto the first coordinate, and generated by the plot
$p:\matR^2\ni(u,v)\mapsto(u,0,|v|)\in\matR^3=V$. Thus, the diffeology of $V$, as well as its structure as the total space of this pseudo-bundle, corresponds to its presentation as the direct product 
$\matR\times\matR^2$; this product diffeology comes from the standard diffeology on $\matR=X$ and the vector space diffeology on $\matR^2$ generated by the plot $v\mapsto(0,|v|)$. The pseudo-bundle
map $\pi$ is just the projection onto the first factor.

Since the diffeology on $V$ is a product diffeology, the pseudo-bundle in question is trivial (diffeologically), so finding a pseudo-metric is easy; we can choose one on the fibre and set it to be the chosen one 
on any other fibre. Thus, the question reduces itself to one regarding pseudo-metrics, in the sense of vector spaces, on $\matR^2$ (we write it in the $(y,z)$-coordinates coming from the ambient $\matR^3=V$) 
with diffeology generated by $v\mapsto(0,|v|)$. This we have seen already, in the sense that each smooth linear map $f$ on this space is given by $(y,z)\mapsto ay$ for some constant $a$; by abuse of notation 
we can then write $f$ as $f=ae^2$ (making an implicit reference to the canonical dual basis of the ambient $\matR^3$ and the same direct product decomposition of its dual space). It follows then that any 
pseudo-metric $g(x)$ on a given fibre $\pi^{-1}(x)$ has form $g(x)=a(x)e^2\otimes e^2$, where $a:\matR\to\matR$ is a smooth everywhere positive function.
\end{example}

\subsection{Non-existence of pseudo-metrics}

It is in fact rather easy to show why a pseudo-metric does not always exist; one of the simplest examples of a non-locally trivial pseudo-bundle suffices for that.

\begin{example}\label{no:pseudometric:exists:ex}
Let $V=\matR^2$ and $X=\matR$; let $\pi:V\to X$ be the projection on the $x$-axis, $\pi(x,y)=x$. Endow $X$ with the standard diffeology, and endow $V$ with the pseudo-bundle diffeology generated by the
plot $p:\matR^2\to V$ defined by $p(u,v)=(u,u|v|)$.
\end{example}

\begin{lemma}\label{there:is:no:pseudometric:lem}
If $\pi:V\to X$ is the pseudo-bundle of Example \ref{no:pseudometric:exists:ex}, it does not admit a pseudo-metric.
\end{lemma}

\begin{proof}
Observe that, outside of the origin of $X$, the diffeology on the fibre is the vector space diffeology on $\matR$ generated by the plot $v\mapsto|v|$, while over the origin the fibre is $\matR$ with the standard 
diffeology. This implies, in particular, that the fibre over a point $x\neq 0$ has trivial dual, while the dual of a fibre over $x=0$ is the standard $\matR$.

Let us now consider the dual bundle. Since from the topological point of view the pseudo-bundle $V$ is a trivial vector bundle, and its vector bundle structure corresponds to the direct product decomposition 
$\matR\times\matR$, we can represent each element $v^*\in V^*$ of the dual bundle $V^*$ in the form $(x,f(x)e^2)$, where $x=\pi^*(v^*)$ and $f(x)e^2$ describes the action of $v^*$ on the fibre 
$\pi^{-1}(x)=\{(x,y)|y\in\matR\}$ (any linear map on this fibre is a multiple of $e^2$ given by $e^2(x,1)=1$; $f(x)$ is the corresponding coefficient).

Now, by definition, a pseudo-metric is a smooth section of $V^*\otimes V^*$; by the same reasoning as above, it can be written as $x\mapsto(x,f(x)e^2\otimes e^2)$ for some function $f:\matR\to\matR$. Notice 
that, since the tensor product $V^*\otimes V^*$ is fibrewise, it is zero everywhere except the origin (over which it is the standard $\matR$). Thus, $f(x)$ is a positive multiple of the $\delta$-function, namely, the 
function $\delta:\matR\to\matR$ given by $\delta(0)=1$ and $\delta(x)=0$ for $x\neq 0$. Thus, let us see whether $g:x\mapsto(x,\delta(x)e^2\otimes e^2)$ is, or is not, a smooth section of $V^*\otimes V^*$.

By linearity and the fact that $V^*\otimes V^*=(V\otimes V)^*$, it suffices to establish the smoothness of the evaluation of $g$ on $q\otimes s$, where $q,s:U\to V$ are an arbitrary pair of plots of $V$ (such that 
$q_1(u)=s_1(u)$). Notice that we can choose $q(u)=s(u)=(u,1)$ for $u\in\matR$ (this is a plot for the pseudo-bundle diffeology). Then the evaluation in question is the function $u\mapsto\delta(u)$, and so, 
obviously, it is not smooth. Since this was the unique possibility, we conclude that no pseudo-metric exists on this pseudo-bundle.
\end{proof}

Lemma \ref{there:is:no:pseudometric:lem} thus shows that a non locally trivial pseudo-bundle may well not admit a pseudo-metric (in the concluding section of this paper we mention a way to rectify this 
specific situation, via a concept of a \emph{$\delta$-metric}), although we would not be ready to describe the totality of situations for non-existence. Likewise, as far as existence matters are concerned, we 
are not able to give a complete characterization of the class of pseudo-bundles that admit a pseudo-metric; the discussion that follows is mostly restricted to the case of pseudo-bundles that are obtained by 
gluing of several (a finite number of) copies of diffeological trivial bundles.

\section{Pseudo-metrics and gluing}

Continuing from the previous section, we turn to the question which in full generality is stated as follows: given a gluing of two pseudo-bundles endowed with a pseudo-metric, when does there exist a
natural pseudo-metric on the result of gluing?

For the duration of this section we fix the following notation: $\pi_1:V_1\to X_1$ and $\pi_2:V_2\to X_2$ are two diffeological vector pseudo-bundles, carrying pseudo-metrics $g_1:X_1\to V_1^*\otimes V_1^*$ 
and $g_2:X_2\to V_2^*\otimes V_2^*$, respectively. Furthermore, we are given a gluing between the two pseudo-bundles, along the maps $f:X_1\supset Y\to X_2$ and $\tilde{f}:\pi_1^{-1}(Y)\to\pi_2^{-1}(f(Y))$. 
Recall that both maps are smooth, and, moreover, $\tilde{f}$ is linear on each fibre in its domain of definition.

\subsection{Compatibility for pseudo-metrics and the induced gluing}

A natural compatibility, with respect to gluing, condition for pseudo-metrics can be stated with only obvious geometric considerations in mind. Specifically, for all $y\in Y$ and for all $v_1,v_2\in\pi_1^{-1}(Y)$ 
we must have
$$g_1(y)(v_1,v_2)=g_2(f(y))(\tilde{f}(v_1),\tilde{f}(v_2)).$$

One may wonder then, whether this notion corresponds to some specific instance of the $(f,g)$-compatibility of smooth maps, introduced in Section \ref{f-g:compatibility:sect}, and if the map 
$\mathcal{F}_{Z_1,Z_2}$ considered in Section \ref{f-g:gluing:of:maps:sect} could be employed to assign to our two pseudo-metrics $g_1$ and $g_2$ one on the pseudo-bundle $\pi_1\cup_{(\tilde{f},f)}\pi_2$ 
obtained by gluing. The answer is easily positive to the former:

\begin{lemma}
Suppose that $f$ is invertible. Then the pseudo-metrics $g_1$ and $g_2$ are compatible if and only if they are $(f^{-1},\tilde{f}^*\otimes\tilde{f}^*)$-compatible.
\end{lemma}

\begin{proof}
Being $(f^{-1},\tilde{f}^*\otimes\tilde{f}^*)$-compatible for $g_1$ and $g_2$ means $g_1\circ f^{-1}=(\tilde{f}^*\otimes\tilde{f}^*)\circ g_2$ wherever defined, and it is obvious that, considered pointwise, this 
equality amounts to $g_1(f^{-1}(y'))(v_1,v_2)=g_2(y')(\tilde{f}(v_1),\tilde{f}(v_2))$. Since $f$ is invertible, this is the same as the above definition of compatibility for $g_1$ and $g_2$.
\end{proof}

Thus, it does make sense to speak of $\mathcal{F}_{V_2^*\otimes V_2^*,V_1^*\otimes V_1^*}(g_2,g_1)$, \emph{i.e.}, the map $g_2\cup_{(f^{-1},\tilde{f}^*\otimes\tilde{f}^*)}g_1$, and we discuss it in the section 
that immediately follows. However, \emph{a priori} it is not a pseudo-metric (in fact, it does not even have the form of one).

\subsection{The map $g_2\cup_{(f^{-1},\tilde{f}^*\otimes\tilde{f}^*)}g_1$, and the switch map $\varphi_{X_1\leftrightarrow X_2}$}

What the concluding sentence of the previous section means precisely, is that the map $g_2\cup_{(f^{-1},\tilde{f}^*\otimes\tilde{f}^*)}g_1$ is defined  between the following spaces:
$$X_2\cup_{f^{-1}}X_1\to(V_2^*\otimes V_2^*)\cup_{\tilde{f}^*\otimes\tilde{f}^*}(V_1^*\otimes V_1^*).$$ Thus, it is not a pseudo-metric, at least not formally; to become one, it must be pre- and post-composed 
with appropriate diffeomorphisms, if such exist.

Indeed, a pseudo-metric on the pseudo-bundle $\pi_1\cup_{(\tilde{f},f)}\pi_2:V_1\cup_{\tilde{f}}V_2\to X_1\cup_f X_2$ is a map of form
$$X_1\cup_f X_2\to(V_1\cup_{\tilde{f}}V_2)^*\otimes(V_1\cup_{\tilde{f}}V_2)^*.$$ Now, since $f$ is a diffeomorphism of its domain with its image, there is a natural diffeomorphism between $X_1\cup_f X_2$ 
and $X_2\cup_{f^{-1}}X_1$, so the real issue regards the existence of a diffeomorphism between $(V_1\cup_{\tilde{f}}V_2)^*\otimes(V_1\cup_{\tilde{f}}V_2)^*$ and
$(V_2^*\otimes V_2^*)\cup_{\tilde{f}^*\otimes\tilde{f}^*}(V_1^*\otimes V_1^*)$. We refer to its existence as the \emph{(gluing-dual) commutativity condition} and state it in a precise form in the section that 
follows. Prior to doing that, we need one further notion.

Specifically, recall the switch map $\varphi$, defined immediately prior to Proposition \ref{diffeological:dual:is:standard:dual:prop}. We now introduce its more general counterpart.

\begin{defn}
Let $X_1$ and $X_2$ be two diffeological spaces, and let $f:X_1\supset Y\to X_2$ be a map that is a diffeomorphism with its image. The \textbf{switch map} 
$$\varphi_{X_1\leftrightarrow X_2}:X_1\cup_f X_2\to X_2\cup_{f^{-1}}X_1$$ is the map satisfying the following conditions: $(\varphi_{X_1\leftrightarrow X_2})\circ i_1'$, where $i_1':X_1\hookrightarrow(X_1\sqcup X_2)\to X_1\cup_f X_2$, coincides with the composition $X_1\hookrightarrow(X_2\sqcup X_1)\to X_2\cup_{f^{-1}}X_1$, while $(\varphi_{X_1\leftrightarrow X_2})\circ i_2$ 
coincides with the composition$X_2\hookrightarrow(X_2\sqcup X_1)\to X_2\cup_{f^{-1}}X_1$.
\end{defn}

The following lemma is quite obvious.

\begin{lemma}
The switch map $\varphi_{X_1\leftrightarrow X_2}$ is uniquely defined and is a diffeomorphism.
\end{lemma}

\subsection{The gluing-dual commutativity condition}

This issue has already been discussed in the section 2.4; we now simply give its formulation appropriate for the current context. Essentially, the commutativity condition is satisfied if
$$(V_1\cup_{\tilde{f}}V_2)^*=V_2^*\cup_{\tilde{f}^*}V_1^*.$$ We have discussed already the fact that essentially, this condition must be stated in this form and cannot be replaced by a simpler version. Thus, 
when we say that \textbf{the (gluing-dual) commutativity condition is fulfilled} (or just, commutativity condition, when it is clear from the context which commutativity we are referring to) when we are guaranteed 
the existence of a diffeomorphism
$$\psi_{\cup\leftrightarrow *}:(V_1\cup_{\tilde{f}}V_2)^*\to V_2^*\cup_{\tilde{f}^*}V_1^*$$ that covers the switch map $\varphi_{X_1\leftrightarrow X_2}$.

\subsection{The commutative case}

In this section we assume that the gluing-dual commutativity condition is satisfied; it is then quite easy to see that there is a canonically defined pseudo-metric on the glued bundle, induced by any pair of 
compatible pseudo-metrics on the factors. Furthermore, this allows for the existence of a smooth map between the relevant spaces of pseudo-metrics, as explained below.

\subsubsection{Constructing the induced pseudo-metric}

Even without the assumption of gluing-dual commutativity, we can always obtain the following map:
$$g_2\cup_{(f^{-1},\tilde{f}^*\otimes\tilde{f}^*)}g_1:X_2\cup_{f^{-1}}X_1\to(V_2^*\otimes V_2^*)\cup_{\tilde{f}^*\otimes\tilde{f}^*}(V_1^*\otimes V_1^*).$$ As has been explained above, since the gluing 
commutes with tensor product, we have the following equality for the range of this map:
$$(V_2^*\otimes V_2^*)\cup_{\tilde{f}^*\otimes\tilde{f}^*}(V_1^*\otimes V_1^*)=(V_2^*\cup_{\tilde{f}^*}V_1^*)\otimes(V_2^*\cup_{\tilde{f}^*}V_1^*);$$ this is always true. But by the assumptions in the case we 
are considering, we also have
$$(V_2^*\cup_{\tilde{f}^*}V_1^*)\otimes(V_2^*\cup_{\tilde{f}^*}V_1^*)=(V_1\cup_{\tilde{f}}V_2)^*\otimes(V_1\cup_{\tilde{f}}V_2)^*.$$ Thus, we obtain a diffeomorphism
$$\Psi_{\cup,*}:(V_2^*\otimes V_2^*)\cup_{\tilde{f}^*\otimes\tilde{f}^*}(V_1^*\otimes V_1^*)\to(V_1\cup_{\tilde{f}}V_2)^*\otimes(V_1\cup_{\tilde{f}}V_2)^*.$$
Finally, we define
$$\tilde{g}:X_1\cup_f X_2\to(V_1\cup_{\tilde{f}}V_2)^*\otimes(V_1\cup_{\tilde{f}}V_2)^*$$ by setting
$$\tilde{g}=\Psi_{\cup,*}\circ(g_2\cup_{(f^{-1},\tilde{f}^*\otimes\tilde{f}^*)}g_1)\circ\varphi_{X_1\leftrightarrow X_2}.$$

\begin{thm}\label{glued:pseudo:metric:commutative:thm}
Let $\pi_1:V_1\to X_1$ and $\pi_2:V_2\to X_2$ be two finite-dimensional diffeological vector pseudo-bundles, and let $(\tilde{f},f)$ be a gluing between them given by a smooth invertible 
$f:X_1\supset Y\to X_2$ and its smooth fibrewise linear lift $\tilde{f}$. Suppose that $(\pi_2^{-1}(f(Y))^*$ and $(\pi_1^{-1}(Y))^*$ are fibrewise diffeomorphic. Finally, assume that for $i=1,2$ there exists a 
pseudo-metric $g_i$ on $X_i$ such that $g_1$ and $g_2$ are $(f,\tilde{f})$-compatible. Then the map
$$\tilde{g}=\Psi_{\cup,*}\circ(g_2\cup_{(f^{-1},\tilde{f}^*\otimes\tilde{f}^*)}g_1)\circ\varphi_{X_1\leftrightarrow X_2},$$
is a pseudo-metric on the pseudo-bundle $\pi_1\cup_{(\tilde{f},f)}\pi_2:V_1\cup_{\tilde{f}}V_2\to X_1\cup_f X_2$.
\end{thm}

\begin{proof}
That $\tilde{g}(x)$ is a bilinear symmetric form on the fibre over each $x\in X_1\cup_f X_2$ follows from the construction (and is explained prior to the statement of the theorem); we should check that 
$\tilde{g}$ is smooth, and that $\tilde{g}(x)$ has the maximal possible rank for the corresponding fibre. Since $\Psi_{\cup,*}$ and $\varphi_{X_1\leftrightarrow X_2}$ are diffeomorphisms, this follows from the 
corresponding properties of $g_2\cup_{(f^{-1},\tilde{f}^*\otimes\tilde{f}^*)}g_1$. Now, the latter map is smooth, for the gluing diffeologies on its domain and co-domain, because because it is obtained by
diffeological gluing of two smooth maps. Its rank should be checked only over the domain of gluing, otherwise it coincides with that of either $g_1$ or $g_2$ (and so is indeed the maximal possible for the
fibre in question).

Let $y\in i_2(f(Y))\subset V_1\cup_{\tilde{f}}V_2$; the fibre $(\pi_1\cup_{(\tilde{f},f)}\pi_2)^{-1}(y)$ is by construction (diffeomorphic to) $\pi_2^{-1}(i_2^{-1}(y))$. The value
$\tilde{g}(y)\in(V_1\cup_{\tilde{f}}V_2)^*\otimes(V_1\cup_{\tilde{f}}V_2)^*$ is obtained as $\tilde{g}(y)=\Psi_{\cup,*}(g_2(i_2^{-1}(y)))=\Psi_{\cup,*}(g_1(i_1^{-1}(f^{-1}(y))))$ (notice that, since $f$ is invertible, 
we extend $i_1$ to the induction $X_1\hookrightarrow X_1\cup_f X_2$ in an obvious way). Notice that $g_2(i_2^{-1}(y))$ is formally an element of $V_2^*\otimes V_2^*$ belonging to the fibre over 
$i_2^{-1}(y)$. It however represents an element of the domain of definition $(V_2^*\otimes V_2^*)\cup_{\tilde{f}^*\otimes\tilde{f^*}}(V_1^*\otimes V_1^*)$ of $\Psi_{\cup,*}$, where it is identified to
$(\tilde{f}^*\otimes\tilde{f}^*)(g_1(i_1^{-1}(f^{-1}(y))))$; this is well-defined by the compatibility of $g_1$ and $g_2$.

Finally, since $\Psi_{\cup,*}$ is a diffeomorphism, we can deduce that the rank of $\tilde{g}(y)$ is equal to that of $g_2(i_2^{-1}(y))$. Since the latter is by assumption a pseudo-metric on 
$\pi_2^{-1}(i_2^{-1}(y))$ (in the sense of diffeological vector spaces, of course), it achieves the maximal rank possible on the space in question, that is, the rank equal to $\dim((\pi_2^{-1}(i_2^{-1}(y)))^*)$.
\end{proof}

\subsection{The noncommutative case}

The absence of commutativity between the operation of gluing and that of taking the dual pseudo-bundle does not necessarily preclude the existence of an induced pseudo-metric on the result of a given
gluing. Indeed, one can define such directly, as we do below, using the fact that every fibre of $V_1\cup_{\tilde{f}}V_2$ is identical to one of either $V_1$ or $V_2$, and thus, by construction, the analogous 
statement is true for every fibre of $(V_1\cup_{\tilde{f}}V_2)^*\otimes(V_1\cup_{\tilde{f}}V_2)^*$.

Let us now describe the pseudo-metric thus obtained. Suppose that we are given two diffeological vector pseudo-bundles $\pi_1:V_1\to X_1$ and $\pi_2:V_2\to X_2$, a gluing of the former to the latter along
an appropriate pair $(\tilde{f},f)$ of maps, and two compatible pseudo-metrics $g_1$ and $g_2$ on $V_1$ and $V_2$ respectively. We now define a section
$$\tilde{g}:X_1\cup_f X_2\to(V_1\cup_{\tilde{f}}V_2)^*\otimes(V_1\cup_{\tilde{f}}V_2)^*$$ by setting its value on $x_1\in i_1(X_1\setminus Y)$ to be
$$\tilde{g}(x_1):=g_1(x_1)\circ(j_1^{-1}\otimes j_1^{-1})\in (V_1\cup_{\tilde{f}}V_2)^*\otimes(V_1\cup_{\tilde{f}}V_2)^*,$$ and then
$$\tilde{g}(x_2):=g_2(x_2)\circ(j_2^{-1}\otimes j_2^{-1})\in (V_1\cup_{\tilde{f}}V_2)^*\otimes(V_1\cup_{\tilde{f}}V_2)^*$$ for $x_2\in i_2(X_2)$. The map $\tilde{g}$ is thus well-defined (simply because $i_1$ 
and $i_2$ are inductions whose images form a disjoint cover of $X_1\cup_f X_2$, and the same holds for $j_1,j_2$ and $V_1\cup_{\tilde{f}}V_2$), and pointwise produces a pseudo-metric on the relevant 
fibre. It remains to show that $\tilde{g}$ is smooth as a map $X_1\cup_f X_2\to(V_1\cup_{\tilde{f}}V_2)^*\otimes(V_1\cup_{\tilde{f}}V_2)^*$.

\begin{thm}\label{glued:pseudo:metric:noncommutative:thm}
Let $\pi_1:V_1\to X_1$ and $\pi_2:V_2\to X_2$ be two finite-dimensional diffeological vector pseudo-bundles, let $(\tilde{f},f)$ be a gluing between them, and let $g_1$ and $g_2$ be pseudo-metrics on 
$V_1$ and, respectively, $V_2$ compatible with respect to the gluing. Define $\tilde{g}:X_1\cup_f X_2\to(V_1\cup_{\tilde{f}}V_2)^*\otimes(V_1\cup_{\tilde{f}}V_2)^*$ as
$$\tilde{g}(x)=\left\{\begin{array}{ll}
((j_1^{-1})^*\otimes(j_1^{-1})^*)\circ g_1(x) & \mbox{for }x\in i_1(X_1\setminus Y) \\
((j_2^{-1})^*\otimes(j_2^{-1})^*)\circ g_2(x) & \mbox{for }x\in i_2(X_2).\end{array}\right.$$ Then $\tilde{g}$ is a pseudo-metric on the pseudo-bundle
$\pi_1\cup_{(\tilde{f},f)}\pi_2:V_1\cup_{\tilde{f}}V_2\to X_1\cup_f X_2$.
\end{thm}

\begin{proof}
Consider a plot $p:U\to X_1\cup_f X_2$ of $X_1\cup_f X_2$. By definition of the gluing diffeology such a plot locally lifts either to a plot $p_2:U\to X_2$ of $X_2$, and then we have $p=i_2\circ p_2$, or to a 
plot $p_1:U\to X_1$ of $X_1$ and then $p$ has the following form:
$$p(u)=\left\{\begin{array}{ll}
i_1(p_1(u)) & \mbox{for }u\in U\mbox{ such that }p_1(u)\in(X_1\setminus Y),\\
i_2(f(p_1(u))) & \mbox{for }u\in U\mbox{ such that }p_1(u)\in Y.\end{array}\right.$$ We need to show that the composition $\tilde{g}(p(u))$ is a plot of the tensor product
$(V_1\cup_{\tilde{f}}V_2)^*\otimes(V_1\cup_{\tilde{f}}V_2)^*$.

By the definition of the tensor product diffeology and that of the gluing diffeology, we need to show that its evaluation
$$(u,u')\mapsto\tilde{g}(p(u))(q(u'),s(u'))$$ on a pair of plots $q,s:U'\to V_1\cup_{\tilde{f}}V_2$ of $V_1\cup_{\tilde{f}}V_2$ (with $q(u'),s(u')$ belonging to the same fibre for all $u'$) is a smooth function on its 
domain of definition. This domain of definition is a subset of $U\times U'$ consisting precisely of those pairs $(u,u')$ for which $(\pi_1\cup_{(\tilde{f},f)}\pi_2)(q(u'))=(\pi_1\cup_{(\tilde{f},f)}\pi_2)(s(u'))=p(u)$
and carries the subset diffeology relative to this inclusion. Furthermore, the assumption that $q$ and $s$ are plots of $V_1\cup_{\tilde{f}}V_2$ means the following (we give an explicit description for $q$ only). 
As it happens for all plots of a gluing diffeology, $q$ locally lifts either to a plot $q_2$ of $V_2$, in which case $q=j_2\circ q_2$, or to a plot $q_1$ of $V_1$, and then it has the following form:
$$q(u')=\left\{\begin{array}{ll} j_1(q_1(u')) & \mbox{if }\pi_1(q_1(u'))\in X_1\setminus Y,\\
j_2(\tilde{f}(q_1(u'))) & \mbox{if }\pi_1(q_1(u'))\in Y. \end{array}\right.$$

Suppose first that $p$ lifts to a plot $p_1$ (the second case); consider the evaluation $(u,u')\mapsto\tilde{g}(p(u))(q(u'),s(u'))$. We have
$$\tilde{g}(p(u))(q(u'),s(u'))=\left\{\begin{array}{ll} g_1(p_1(u))(q_1(u'),s_1(u')) & \mbox{if }p_1(u)\in X_1\setminus Y,\\
g_2(f(p_1(u)))(\tilde{f}(q_1(u')),\tilde{f}(s_1(u'))) & \mbox{if }p_1(u)\in Y.\end{array}\right.$$ Now, the first expression is the evaluation of $g_1\circ p_1$ at a pair of plots $q_1,s_1$ of $V_1$, while the second 
one equals to
$$g_2(f(p_1(u)))(\tilde{f}(q_1(u')),\tilde{f}(s_1(u')))=g_1(p_1(u))(q_1(u'),s_1(u'))$$ by the compatibility condition. Thus, $\tilde{g}(p(u))(q(u'),s(u'))$ coincides, for all $(u,u')$ for which it is defined, with the
evaluation of $g_1\circ p_1$ on a pair of plots $q_1,s_1$ of $V_1$, so it is a smooth function by the smoothness of $g_1$.

Finally, if we suppose that $p$ lifts to a plot $p_2$ of $X_2$, then in the domain of the evaluation function, the two plots $q$ and $s$ must also lift to plots $q_2$ and $s_2$ of $V_2$ (not those of $V_1$). 
Therefore the result of the evaluation function is 
$$\tilde{g}(p(u))(q(u'),s(u'))=g_2(p_2(u))(q_2(u'),s_2(u')).$$ Since the latter coincides with the value of the evaluation function for the pseudo-metric $g_2$ on plots $p_2$, $q_2$, and $s_2$, we conclude 
that $(u,u')\mapsto\tilde{g}(p(u))(q(u'),s(u'))$ is smooth, whence the final conclusion of the theorem.
\end{proof}

\subsection{Comparison of commutative and noncommutative cases}

It is natural to wonder at this point whether the two versions of the pseudo-metric $\tilde{g}$ on $\pi_1\cup_{(\tilde{f},f)}\pi_2:V_1\cup_{\tilde{f}}V_2\to X_1\cup_f X_2$ provided respectively by Theorem
\ref{glued:pseudo:metric:commutative:thm} and by Theorem \ref{glued:pseudo:metric:noncommutative:thm}, coincide in the case when both are applicable (that is, in the commutative case). We now show 
that they always do (this is a direct consequence of the construction).

\begin{thm}\label{two:pseudo:metrics:coincide:when:defined:thm}
Let $\pi_i:V_i\to X_i$ for $i=1,2$ be two diffeological vector pseudo-bundles, and let $(\tilde{f},f)$ be a usual pair of maps defining gluing of $V_1$ to $V_2$. Suppose that there is a diffeomorphism 
$(V_1\cup_{\tilde{f}}V_2)^*\to V_2^*\cup_{\tilde{f}^*}V_1^*$ covering the switch map $\varphi_{X_1\leftrightarrow X_2}$. Finally, let $g_i$ be a pseudo-metric on $\pi_i:V_i\to X_i$ and assume that $g_1$ and 
$g_2$ are $(f,\tilde{f})$-compatible. Then the commutative and the noncommutative versions of $\tilde{g}$ yield the same pseudo-metric on 
$\pi_1\cup_{(\tilde{f},f)}\pi_2:V_1\cup_{\tilde{f}}V_2\to X_1\cup_f X_2$.
\end{thm}

\begin{proof}
Let us compare the two maps pointwise, denoting the commutative version by $\tilde{g}'$ and the noncommutative one by $\tilde{g}''$. Let $x_1\in i_1(X_1\setminus Y)$, and let
$v_1,v_2\in(\pi_1\cup_{(\tilde{f},f)}\pi_2)^{-1}(x_1)$; note that $v_1,v_2\in j_1(V_1\setminus\pi_1^{-1}(Y))$. Starting with $\tilde{g}'$, we obtain
$$\tilde{g}'(x_1)(v_1,v_2)=g_1(i_1^{-1}(x_1))(j_1^{-1}(v_1),j_1^{-1}(v_2))=\left(g_1(i_1^{-1}(x_1))\circ j_1^{-1}\otimes j_1^{-1}\right)(v_1,v_2)=\tilde{g}''(x_1).$$ Similarly, if $x_2\in i_2(X_2)$ and 
$v_1,v_2\in(\pi_1\cup_{(\tilde{f},f)}\pi_2)^{-1}(x_2)$ then $v_1,v_2\in j_2(V_2)$, and we obtain
$$\tilde{g}'(x_2)(v_1,v_2)=g_2(i_2^{-1}(x_2))(j_2^{-1}(v_1),j_2^{-1}(v_2))=\left(g_2(i_2^{-1}(x_2))\circ j_2^{-1}\otimes j_2^{-1}\right)(v_1,v_2)=\tilde{g}''(x_2).$$
\end{proof}

We conclude saying that the commutative version, when defined, is in some cases more convenient, being a kind of functorial construction; the noncommutative version is more direct and more generally
applicable.

\subsection{Using the gluing construction to prove existence}

We conclude this section with observing that the gluing construction and the corresponding procedure for obtaining a pseudo-metric on the resulting pseudo-bundle (see Theorem
\ref{glued:pseudo:metric:noncommutative:thm}) might allow for proving existence of pseudo-metrics on a rather wide class of pseudo-bundles, such as all those obtained by gluing together standard bundles 
(projections of form $\matR^n=\matR^k\times\matR^{n-k}\to\matR^k$), as long as compatible pseudo-metrics can be found on the factors. Since we are not ready to give a comprehensive statement to this 
matter, we limit ourselves to this brief remark.

\section{The spaces of pseudo-metrics $\mathcal{G}(V_1,X_1)$ and $\mathcal{G}(V_2,X_2)$, and the space $\mathcal{G}(V_1\cup_{\tilde{f}}V_2,X_1\cup_f X_2)$}

For any finite-dimensional diffeological vector pseudo-bundle $\pi:V\to X$, let $\mathcal{G}(V,X)$ stand for the space of all pseudo-metrics on this pseudo-bundle. This space is naturally endowed with the 
functional diffeology, and so is a diffeological space.

The above-considered assignment of a new pseudo-metric $\tilde{g}$, for every pair of compatible pseudo-metrics $g_1$ and $g_2$, can be seen as a map defined on some subset of
$\mathcal{G}(V_1,X_1)\times\mathcal{G}(V_2,X_2)$ and taking values in $\mathcal{G}(V_1\cup_{\tilde{f}}V_2,X_1\cup_f X_2)$. In this section we consider the smoothness of this map (for the diffeologies
indicated).

\subsection{The map $\mathcal{P}$}

Let $\pi_1:V_1\to X_1$ and $\pi_2:V_2\to X_2$ be two finite-dimensional diffeological vector pseudo-bundles, let $f:X_1\supset Y\to X_2$ be a smooth map, and let $\tilde{f}:\pi_1^{-1}(Y)\to\pi_2^{-1}(f(Y))$ be 
a smooth lift of $f$ that is linear on each fibre. Suppose that $\mathcal{G}(V_i,X_i)\neq\emptyset$ for $i=1,2$, and assume that there exists at least one pair of compatible pseudo-metrics $g_1,g_2$ with 
$g_i\in\mathcal{G}(V_i,X_i)$; note that the construction of $\tilde{g}$ out of $g_1,g_2$ that is described in the previous section allows to conclude immediately that the set 
$\mathcal{G}(V_1\cup_{\tilde{f}}V_2,X_1\cup_f X_2)$ is also non-empty.

Consider now the subset $\mathcal{G}(V_1,X_1)\times_{comp}\mathcal{G}(V_2,X_2)$ of the direct product $\mathcal{G}(V_1,X_1)\times\mathcal{G}(V_2,X_2)$; this subset consists of all pairs of compatible 
pseudo-metrics and is endowed with the subset diffeology relative to the product diffeology on $\mathcal{G}(V_1,X_1)\times\mathcal{G}(V_2,X_2)$ (this product diffeology, in turn, is relative to the functional
diffeologies on $\mathcal{G}(V_1,X_1)$ and $\mathcal{G}(V_2,X_2)$).

Let us define the following map:
$$\mathcal{P}:\mathcal{G}(V_1,X_1)\times_{comp}\mathcal{G}(V_2,X_2)\to\mathcal{G}(V_1\cup_{\tilde{f}}V_2,X_1\cup_f X_2)$$ acting by
$$\mathcal{P}(g_1,g_2)=\tilde{g},$$ where $\tilde{g}$ is associated (see Theorem \ref{glued:pseudo:metric:commutative:thm}) to $g_2\cup_{(f^{-1},\tilde{f}^*\otimes\tilde{f}^*)}g_1$ in the gluing-dual 
commutative case and is given by Theorem \ref{glued:pseudo:metric:noncommutative:thm} in the corresponding non-commutative case. We shall now show that this map is smooth for the diffeologies on its 
domain and its range indicated above.

\subsection{Smoothness with the commutativity assumption}

Although the pseudo-metric $\tilde{g}$ obtained under the assumption of the gluing-dual commutativity is a partial case of the one obtained without such assumption, we do give a separate proof of the
smoothness of $\mathcal{P}$ in this case (it is quite different, and also brief).

\begin{thm}
The map $\mathcal{P}$ is smooth.
\end{thm}

\begin{proof}
The map $\mathcal{P}$ can be written as the composition of three maps, specifically: the map
$$\mathcal{G}(V_1,X_1)\times_{comp}\mathcal{G}(V_2,X_2)\to\mathcal{G}(V_2,X_2)\times\mathcal{G}(V_1,X_1),$$ which is just the order change within the direct product; the map
$$\mathcal{F}_{V_2^*\otimes V_2^*,V_1^*\otimes V_1^*}:C^{\infty}(X_2,V_2^*\otimes V_2^*)\times_{comp}C^{\infty}(X_1,V_1^*\otimes V_1^*)\to C^{\infty}(X_2\cup_{f^{-1}}X_1,(V_2^*\otimes
V_2^*)\cup_{\tilde{f}^*\otimes\tilde{f}^*}(V_1^*\otimes V_1^*))$$ (where the compatibility is considered with respect to the maps $f^{-1}$ and $\tilde{f}^*\otimes\tilde{f}^*$), which is a particular case of the map 
of Theorem \ref{f-g-compatible:gluing:smooth:thm}; and the map 
$$C^{\infty}(X_2\cup_{f^{-1}}X_1,(V_2^*\otimes V_2^*)\cup_{\tilde{f}^*\otimes\tilde{f}^*}(V_1^*\otimes V_1^*))\to C^{\infty}(X_2\cup_{f^{-1}}X_1,(V_1\cup_{\tilde{f}}V_2)^*\otimes(V_1\cup_{\tilde{f}}V_2)^*),$$
which acts by post-composing with the fixed map $\Psi_{\cup,*}$.

The first of these maps is smooth by definition of the product diffeology; and so is the second, by Theorem \ref{f-g-compatible:gluing:smooth:thm}. That the third map is smooth, can easily be deduced from the 
definition of a functional diffeology and the fact that it is a composition with a fixed map, which can be seen as the image of a constant plot.\footnote{We are making an implicit reference to the fact that, if $X,Y,Z$
are diffeological spaces, and each of $C^{\infty}(X,Y)$, $C^{\infty}(Y,Z)$, $C^{\infty}(X,Z)$ is considered with the functional diffeology, then the composition map 
$C^{\infty}(X,Y)\times C^{\infty}(Y,Z)\to C^{\infty}(X,Z)$ is smooth (for the product diffeology on its domain of definition), see \cite{iglesiasBook}, Section 1.59.}
\end{proof}

\subsection{Smoothness without the commutativity assumption}

In this absence of the gluing-dual commutativity, the description of the pseudo-metric $\tilde{g}=\mathcal{P}(g_1,g_2)$ is pointwise, so we give a direct proof.

\begin{thm}
For any two $(f,\tilde{f})$-compatible pseudo-metrics $g_1$ and $g_2$ on the pseudo-bundles $\pi_1:V_1\to X_1$ and $\pi_2:V_2\to X_2$ consider the pseudo-metric $\tilde{g}$ on the pseudo-bundle
$\pi_1\cup_{(\tilde{f},f)}\pi_2:V_1\cup_{\tilde{f}}V_2\to X_1\cup_f X_2$ defined, for every $x\in X_1\cup_f X_2$ and $v,w\in(\pi_1\cup_{(\tilde{f},f)}\pi_2)^{-1}(x)$ as
$$\tilde{g}(x)(v_1,v_2)=\left\{\begin{array}{ll} g_1(i_1^{-1}(x))(j_1^{-1}(v),j_1^{-1}(w)) & \mbox{if }x\in\mbox{Im}(i_1),\\
g_2(i_2^{-1}(x))(j_2^{-1}(v),j_2^{-1}(w)) & \mbox{if }x\in\mbox{Im}(i_2). \end{array}\right.$$ Then the map
$$\mathcal{P}:\mathcal{G}(V_1,X_1)\times_{comp}\mathcal{G}(V_2,X_2)\to\mathcal{G}(V_1\cup_{\tilde{f}}V_2,X_1\cup_f X_2)$$ defined by $\mathcal{P}(g_1,g_2)=\tilde{g}$ is smooth.
\end{thm}

\begin{proof}
We must show that the composition of $\mathcal{P}$ with an arbitrary plot $p$ of $\mathcal{G}(V_1,X_1)\times_{comp}\mathcal{G}(V_2,X_2)$ is a plot of $\mathcal{G}(V_1\cup_{\tilde{f}}V_2,X_1\cup_f X_2)$. 
By definitions of the subset diffeology and the product one, $p$ writes as a pair of form $(p_1,p_2)$, where each $p_i:U\to\mathcal{G}(V_i,X_i)$ is a plot of $\mathcal{G}(V_i,X_i)$; the diffeology of the latter 
being the subset diffeology relative to the functional diffeology on $C^{\infty}(X_i,V_i^*\otimes V_i^*)$, our assumptions thus amount to precisely the following. First of all, the map 
$\varphi_i:U\times X_i\to V_i^*\otimes V_i^*$ given by $(u,x)\mapsto p_i(u)(x)$ must be smooth; this in turn means that for every plot $q_i:U_i'\to X_i$ the map defined by 
$U\times U_i'\ni(u,u')\mapsto p_i(u)(q_i(u'))\in V_i^*\otimes V_i^*$ must be a plot of $V_i^*\otimes V_i^*$. And this, finally, means that for any two plots $s_i,t_i:U_i''\to V_i$ the map
$$U\times U_i'\times U_i''\ni(u,u',u'')\mapsto p_i(u)(q_i(u'))(s_i(u''),t_i(u'')),$$ defined on the subset of triples $(u,u',u'')$ such that $\pi_i(s_i(u''))=\pi_i(t_i(u''))=q_i(u')$, must be a $\matR$-valued smooth map 
(for the subset diffeology of the appropriate triples).

On the other hand, we need to show that $\mathcal{P}\circ p$ is a plot of $\mathcal{G}(V_1\cup_{\tilde{f}}V_2,X_1\cup_f X_2)$. This again amounts to taking an arbitrary plot $q:U'\to X_1\cup_f X_2$ of
$X_1\cup_f X_2$ and showing that $U\times U'\ni(u,u')\mapsto(\mathcal{P}(p(u)))(q(u'))$ is a plot of $(V_1\cup_{\tilde{f}}V_2)^*\otimes(V_1\cup_{\tilde{f}}V_2)^*$, that is, for any two plots 
$s,t:U''\to V_1\cup_{\tilde{f}}V_2$ of $V_1\cup_{\tilde{f}}V_2$, the map
$$U\times U'\times U''\ni(u,u',u'')\mapsto(\mathcal{P}(p(u)))(q(u'))(s(u''),t(u'')),$$ once again defined precisely for those triples for which 
$(\pi_1\cup_{(\tilde{f},f)}\pi_2)(s(u''))=(\pi_1\cup_{(\tilde{f},f)}\pi_2)(t(u''))=q(u')$, is a $\matR$-valued map smooth for the subset diffeology on the set of such triples.

Recall, finally, that by definition of $\mathcal{P}$ we have
$$\mathcal{P}(p(u))(q(u'))=\left\{\begin{array}{ll} p_1(u)(i_1^{-1}(q(u')))\circ(j_1^{-1}\otimes j_1^{-1}) & \mbox{if }q(u')\in\mbox{Im}(i_1),\\
p_2(u)(i_2^{-1}(q(u')))\circ(j_2^{-1}\otimes j_2^{-1}) & \mbox{if }q(u')\in\mbox{Im}(i_2).\end{array}\right.$$

Now, every plot $q$ of $X_1\cup_f X_2$ locally lifts to either a plot $q_1$ of $X_1$, or to a plot $q_2$ of $X_2$. Suppose first that it lifts to a plot $q_2$; this means that $q=i_2\circ q_2$. Then we have:
$$(\mathcal{P}(p(u)))(q(u'))(s(u''),t(u''))=p_2(u)(q_2(u'))(j_2^{-1}(s(u'')),j_2^{-1}(t(u''))).$$ Recalling that $j_2^{-1}\circ s$ and $j_2^{-1}\circ t$ are plots of $V_2$ (by $j_2$ being an induction), we conclude that 
the evaluation $(u,u',u'')\mapsto(\mathcal{P}(p(u)))(q(u'))(s(u''),t(u''))$ is a smooth map from its domain of definition to $\matR$.

Suppose, finally, that $q$ lifts to a plot $q_1$ of $X_1$; this means that
$$q(u')=\left\{\begin{array}{ll} i_1(q_1(u')) & \mbox{if }q_1(u')\in X_1\setminus Y,\\
i_2(f(q_1(u'))) & \mbox{if }q_1(u')\in Y. \end{array}\right. $$ Accordingly, we obtain
$$\mathcal{P}(p(u))(q(u'))(s(u''),t(u''))=\left\{\begin{array}{ll} p_1(u)(q_1(u'))(j_1^{-1}(s(u'')),j_1^{-1}(t(u''))) & \mbox{if }q_1(u')\in X_1\setminus Y,\\
p_2(u)(f(q_1(u')))(j_2^{-1}(s(u'')),j_2^{-1}(t(u''))) & \mbox{if }q_1(u')\in Y.\end{array}\right.$$ Recall now that $s:U''\to V_1\cup_{\tilde{f}}V_2$ (we consider only one of $s,t$, the reasoning is completely the same 
for the other) is a plot of $V_1\cup_{\tilde{f}}V_2$, so (as usual, we assume that $U''$ is small enough) it lifts either to a plot of $V_1$ or that of $V_2$. Now, since $(\pi_1\cup_{(\tilde{f},f)}\pi_2)(s(u''))=q(u')$,
and we assumed that $q$ lifts to a plot of $X_1$, we see that $s$ lifts to a plot $s_1$ of $V_1$, and furthermore, whenever $q_1(u')\in Y$, we have that $j_2^{-1}(s(u''))=\tilde{f}(s_1(u''))$ (while we have 
$j_1^{-1}(s(u''))=s_1(u'')$ elsewhere). For the same reason, there is a plot $t_1$ of $V_1$ such that $q_1(u')\in Y\Rightarrow j_2^{-1}(t(u''))=\tilde{f}(t_1(u''))$ and $j_1^{-1}(t(u''))=t_1(u'')$ in all other cases.

It remains to observe that, by compatibility of the pseudo-metrics $p_1(u)$ and $p_2(u)$, we have
$$p_2(u)(f(q_1(u')))(j_2^{-1}(s(u'')),j_2^{-1}(t(u'')))=p_2(u)(f(q_1(u')))(\tilde{f}(s_1(u'')),\tilde{f}(t_1(u'')))=$$
$$=p_1(u)(q_1(u'))(s_1(u''),t_1(u'')),$$ which is the same expression that defines $\mathcal{P}(p(u))(q(u'))(s(u''),t(u''))$ for $q_1(u')\in X_1\setminus Y$, and we get the final conclusion from the assumptions 
on $p(u)$.
\end{proof}

We thus can see directly that there is a clear relation between the space of pairs of compatible pseudo-metrics on the factors of gluing, and the space of pseudo-metrics on the result of gluing. Furthermore, it 
is quite clear that the map $\mathcal{P}$ is a diffeomorphism if $\tilde{f}$ is so; in fact, it probably suffices for $\tilde{f}$ to be an induction (checking this thoroughly is left for future work). More in general, 
$\mathcal{P}$ is always injective (this follows simply from the fact that the restriction of $\mathcal{P}(g_1,g_2)$ on each fibres coincides with either $g_1$ or $g_2$, and from the compatibility of these); on the 
other hand, it is not quite clear whether/when it would be surjective (once again, the issue is left for future work).

\section{The induced metric on the dual pseudo-bundle}

Similarly to the standard case, if a given finite-dimensional diffeological vector pseudo-bundle $\pi:V\to X$ admits a pseudo-metric $g$, then there is a natural induced pseudo-metric on its dual pseudo-bundle 
$\pi^*:V^*\to X$.

\paragraph{The natural pairing map induced by a pseudo-metric} Let us define the following pseudo-bundle map $\Phi$ from $V$ to $V^*$. We set, for every $v\in V$, that
$$\Phi(v)=g(\pi(v))(v,\cdot)\in V^*.$$ The restriction of this map to each fibre is the natural pairing via $g$; as a map between diffeological vector spaces, it was already considered in \cite{pseudometric}, where 
it was shown that it is surjective and establishes a diffeomorphism between the dual and a  specific subspace of the domain space.

\begin{lemma}\label{Phi:smooth:linear:maps:bundle:lem}
Let $\pi:V\to X$ be a finite-dimensional diffeological vector pseudo-bundle that admits a pseudo-metric, and let $g:X\to V^*\otimes V^*$ be a pseudo-metric on $V$. Then $\Phi$ is smooth as a map $V\to V^*$.
\end{lemma}

\begin{proof}
In order to show that $\Phi$ is smooth, we need to show that, for any arbitrary plot $p:U\to V$ of $V$ its composition $\Phi\circ p$ is a plot of $V^*$. Since we have $(\Phi\circ p)(u)=g(\pi(p(u)))(p(u),\cdot)$, and 
by definition of the diffeology on $V^*$, it suffices to prove that its evaluation on any other plot $q:U'\to V$ is smooth (as a map from the set of $(u,u')$ such that $\pi(p(u))=\pi(q(u'))$), and this follows from smoothness of $g$.
\end{proof}

\begin{rem}
Notice that the same conclusion holds for any smooth bilinear map $g:X\to V^*\otimes V^*$, not necessarily a pseudo-metric: whether it has, or not, the maximal rank possible for a given fibre, is inessential 
for smoothness of $\Phi$.
\end{rem}

\paragraph{The pseudo-metric on $V^*$} We now make use of the above lemma to obtain a pseudo-metric on the dual pseudo-bundle of a given one, provided it is locally trivial.

\begin{thm}\label{metric:on:dual:simult:thm}
Let $\pi:V\to X$ be a finite-dimensional diffeological vector pseudo-bundle that is locally trivial and admits a pseudo-metric $g$. Then there is a natural induced pseudo-metric $g^*$ on its dual pseudo-bundle 
$\pi^*:V^*\to X$.
\end{thm}

\begin{proof}
Let us define $g^*$ pointwise; recall (\cite{pseudometric}) that for each $x\in X$ the restriction of the map $\Phi$ to the fibre over $x$ is surjective onto the corresponding fibre of $V^*$. We define
$$g^*(x)(\Phi(v),\Phi(w)):=g(x)(v,w)\mbox{ for all }x\in X\mbox{ and }v,w\in\pi^{-1}(x).$$ By the same result, this is well-defined, since whenever $\Phi(x)(v)=\Phi(x)(v')$, the elements $v,v'\in\pi^{-1}(x)$ differ
by an element of the isotropic subspace of $g(x)$. It remains to see that $g^*$ is smooth as a map $X\to(V^*)^*\otimes(V^*)^*$.

Let $x\in X$ be an arbitrary point, and let $X'\ni x$ be a neighborhood of it such that $\pi^{-1}(X')\cong X'\times\pi^{-1}(x)$. Recall (see \cite{pseudometric}) that $\pi^{-1}(x)$, being a finite-dimensional 
diffeological vector space endowed with a pseudo-metric $g(x)$, contains a subspace $V_0$ such that the restriction of $\Phi$ to this subspace is a diffeomorphism $V_0\to(\pi^{-1}(x))^*$. Then the assumption 
of the local triviality this restriction yields a diffeomorphism $\pi^{-1}(X')\supset X'\times V_0\to(\pi^*)^{-1}(X')$, which shows that $g^*=\mbox{\textsc{ev}}(g)(\Phi^{-1}\otimes\Phi^{-1})$ over $X'$, is indeed 
smooth.
\end{proof}

\begin{rem}
We are not able to say whether the assumption of local triviality is truly necessary for the existence of a smooth $g^*$, which is somewhat disappointing. 
\end{rem}

\vspace{1cm}

\noindent University of Pisa \\
Department of Mathematics \\
Via F. Buonarroti 1C\\
56127 PISA -- Italy\\
\ \\
ekaterina.pervova@unipi.it\\

\end{document}